\documentclass[11pt,reqno]{amsart}
\usepackage[utf8]{inputenc}
\usepackage{amsmath}
\usepackage{amsfonts}
\usepackage{amssymb}
\usepackage{amsthm}
\usepackage{hyperref}
\usepackage{multicol}
\usepackage{pictexwd, dcpic}
\usepackage{fancyhdr}
\usepackage{mathrsfs}
\usepackage{latexsym,mathtools}
\usepackage{bm, enumerate}
\usepackage{graphicx}
\usepackage{color,soul}
\usepackage{xcolor}
\usepackage{hyperref}
\usepackage{dsfont}
\usepackage{todonotes}
\usepackage{tikz-cd}
\makeatletter
\@namedef{subjclassname@2020}{%
	\textup{2020} Mathematics Subject Classification}
\makeatother

\newcommand{\beq}{\begin{eqnarray}}
	\newcommand{\eeq}{\end{eqnarray}}
\newcommand{\beqn}{\begin{eqnarray*}}
	\newcommand{\eeqn}{\end{eqnarray*}}
\newcommand{\rar}{\rightarrow}
\newcommand*{\mf}{\mathbf}

\reversemarginpar  

\numberwithin{equation}{section}

\newtheorem{theorem}{\bf Theorem}[section]
\newtheorem{proposition}[theorem]{\bf Proposition}
\newtheorem{corollary}[theorem]{\bf Corollary}
\newtheorem{lemma}[theorem]{\bf Lemma}

\theoremstyle{remark}
\newtheorem{definition}[theorem]{\bf Definition}
\newtheorem{example}[theorem]{\bf Example}
\newtheorem{Non-example}[theorem]{\bf Non-Example}
\newtheorem{question}[theorem]{\bf Question}
\newtheorem{remark}[theorem]{\bf Remark}

\newcommand{\clb}{\mathcal{B}}

\newcommand{\clh}{\mathcal{H}}

\newcommand{\clm}{\mathcal{M}}

\newcommand{\cls}{\mathcal{S}}

\newcommand{\D}{\mathbb{D}}

\newcommand{\C}{\mathbb{C}}

\newcommand{\z}{\bm{z}}
\newcommand{\w}{\bm{w}}

\newcommand*{\Ge}{\geqslant}

\newcommand*{\inp}[2]{\langle{#1},\,{#2} \rangle}
\newcommand*{\Le}{\leqslant}

\def\D{\mathbb{D}}

 \newcommand{\ntwo}{\mathbb N_{\Ge 2}}

\theoremstyle{definition}

\theoremstyle{remark}

\hypersetup{
	colorlinks,
	citecolor=blue,
	linkcolor=blue,
	urlcolor=blue}

\begin{document}
\title[Multiplier varieties and multiplier algebras of CNP Dirichlet series kernels]{Multiplier varieties and multiplier algebras of CNP Dirichlet series kernels}

\author[H. Ahmed]{Hamidul Ahmed}
\address{Department of Mathematics, Indian Institute of Technology Bombay,
Powai, Mumbai, 400076, India}
\email{hamidulahmedslk@gmail.com, 22D0785@iitb.ac.in}

\author[B.K. Das]{B. Krishna Das}
\address{Department of Mathematics, Indian Institute of Technology Bombay,
Powai, Mumbai, 400076, India}
\email{bata436@gmail.com, dasb@math.iitb.ac.in}

\author[C. K. Sahu]{Chaman Kumar Sahu}
\address{Department of Mathematics, Indian Institute of Technology Bombay,
Powai, Mumbai, 400076, India}
\email{sahuchaman9@gmail.com, chamanks@math.iitb.ac.in}


\begin{abstract}
We investigate isometric and algebraic isomorphism problems for multiplier algebras associated with Dirichlet series kernels that possess the complete Nevanlinna–Pick (CNP) property. 
A central aspect of our work is the explicit determination of the multiplier variety associated with each CNP Dirichlet series kernel, via polynomial equations derived from the arithmetic structure of the associated weight and frequency data.
This description of multiplier varieties enables us to classify when the multiplier algebras of a significant class of CNP Dirichlet series kernels are isomorphic, or isometrically isomorphic. 
In this setting, a striking rigidity phenomenon emerges whereby the multiplier algebra determines the kernel up to natural equivalence. The results established for CNP Dirichlet series kernels also extend to classical CNP kernels, yielding new results for the associated multiplier algebras even in the classical setting.
As an application, we resolve an open problem posed by McCarthy and Shalit (\cite{MS}).

\end{abstract}

\maketitle
\section{Introduction}
%

The study of multiplier algebras associated with reproducing kernel Hilbert spaces (RKHS) emerged as a central theme in modern multivariable operator theory, function theory, and operator algebras. The multiplier algebra of an RKHS is a dual operator algebra--typically non-selfadjoint--and hence plays a significant role in the broader theory of non-selfadjoint operator algebras. Moreover, it serves as a natural bridge between operator theory and function theory through its deep connections with interpolation problems, and in many cases, it also links to complex geometry via the underlying analytic varieties.
On the other hand, complete Nevanlinna–Pick (CNP) spaces often arising as subspaces of the Drury–Arveson space have attracted much interest due to their well-structured behavior, in contrast to general reproducing kernel Hilbert spaces. Many fundamental properties of classical spaces, such as the Hardy space and the Drury-Arveson space, naturally extend to the CNP setting. Numerous mathematicians have contributed to aspects of the theory of CNP spaces, including the structure and classification of their multiplier algebras; see, for example, the (necessarily incomplete) list of references \cite{ADP, ABJK, RH, DHS, DRS, Hartz_1, Hartz, KMS, Mccull92, Quiggin, Sa-Sh}. The monograph by Agler and McCarthy (\cite{AM}) is a foundational source on the theory of CNP spaces. 
Among the many types of CNP spaces studied, Hilbert spaces of Dirichlet series whose kernels exhibit the CNP property have emerged as a particularly rich class, connecting techniques from operator theory, function theory in several complex variables, and number theory.
The 
study of the Hilbert space of Dirichlet series gained significant attention through the 
seminal work of Hedenmalm, Lindqvist, and Seip (\cite{HLS}), which determines the multiplier algebra of the Hardy space of Dirichlet series and resolves a completeness problem originally posed by Beurling (\cite{Beurling}). 
Since then, numerous studies (e.g. \cite{AOS, Brevig-thesis, BPS, C-S, GH, MS, Ol, Sa,  Seip}) have explored aspects of their structure, from operator-theoretic and function-theoretic properties to the behavior of their multiplier algebras. The recent monograph \cite{QQ} gives a comprehensive account of Dirichlet series. 

This article is concerned with the isometric and algebraic isomorphism problems for multiplier algebras associated with CNP Dirichlet series kernels. Within this framework, the role of Dirichlet series kernels is twofold. On the one hand they provide concrete models of CNP spaces of Dirichlet series on right half-planes, and on the other their arithmetic structure--encoded by so-called weight and frequency data--determines the geometry of the associated multiplier varieties. Building on the work of \cite{DHS} and \cite{MS}, we determine the multiplier varieties associated with CNP Dirichlet series kernels and investigate the extent to which the underlying arithmetic data determine the associated multiplier algebra up to isometric or algebraic isomorphism.

Our first result is a new characterization of the set of all normalized CNP Dirichlet series kernels defined on a right half-plane. Let us denote by $\mathbb N$ the set of all natural numbers, and set
\[\ntwo:=\{n\in\mathbb N: n\Ge 2\}.\]
 We show (Theorem~\ref{all-cnp-char}) that CNP Dirichlet series kernels defined on some right half-plane $\mathbb H_\rho=\{s \in \mathbb C : \text{Re}(s) > \rho \}$, $\rho\in\mathbb R$, are precisely those of the form
\beqn
K_{\mf b, \mf n}(s,u) = \frac{1}{1-\sum_{j=1}^d b_jn_j^{-s-\overline{u}}},\quad s, u \in \mathbb H_\rho
\eeqn
for some $d\in\mathbb N \cup \{\infty\}$, $d$-tuples $\mf n=(n_j)_{j=1}^d $ of distinct elements in $\ntwo$, and $\mf b=(b_j)_{j=1}^d \in (0, \infty)^d$.
If $d=\infty$, for $K_{\mf b,\mf n}$ to be defined on some right half-plane, it is necessary that 
\[
\sigma_a\Big(\sum_{j=1}^d b_jn_j^{-s}\Big)< \infty,
\]
where we write $\sigma_a(\cdot)$ for the abscissa of absolute convergence of a Dirichlet series (see ~\eqref{sigma-a}).
We do not state this assumption explicitly, but it is understood to be in force whenever $d=\infty$.  
We refer to $\mf b$ and $\mf n$ as the \emph{weight and frequency} data, respectively, associated with the CNP kernel $K_{\mf b,\mf n}$. Since these data uniquely determine a CNP Dirichlet series kernel, \emph{we adopt the convention that $(\mf b,\mf n)$ denotes weight and frequency data of the same length}; that is, $\mf b$ is a $d$-tuple of positive real numbers and $\mf n$ is a $d$-tuple of distinct elements of $\mathbb{N}_{\ge 2}$ for some $d\in \mathbb N\cup \{\infty\}$. We denote by $\mathbb C^d$ the $d$-dimensional complex plane if $d<\infty$, and for $d=\infty$, $\mathbb C^d$ is the Hilbert space of all square summable sequences. The open unit ball in $\mathbb C^d$ is denoted by $\mathbb B_d$. Notice that, in terms of the function
$f_{\mf b, \mf n}:\mathbb H_\rho \rar \mathbb B_d$ defined by
\beq
\label{f-at-s-intro}
f_{\mf b, \mf n}(s) =(\sqrt{b_j} n_j^{-s})_{j=1}^d,\quad s \in \mathbb H_\rho,
\eeq
CNP Dirichlet series kernels are expressible in the standard form: 
\[
K_{\mf b, \mf n}(s,u) = \frac{1}{1-\sum_{j=1}^d b_jn_j^{-s-\overline{u}}}= \frac{1}{1-\langle f_{\mf b, \mf n} (s), f_{\mf b, \mf n}(u) \rangle},\quad s, u \in \mathbb H_\rho. 
\]
This realization identifies the CNP Dirichlet series kernel $K_{\mf b, \mf n}$ as the pullback of the Drury-Arveson kernel
\[
\mathcal S_d(\z,\w)=\frac{1}{1-\langle \z,\w\rangle}_{\C^d},\quad \mf z, \mf w \in \mathbb B_d.
\]
Thus, the multiplier algebra $\mathcal M(\mathcal H(K_{\mf b, \mf n}))$ of the reproducing kernel Hilbert space $\mathcal H(K_{\mf b, \mf n})$ associated with $K_{\mf b, \mf n}$ is isometrically isomorphic to the restriction of the multiplier algebra $\mathcal M_d$ of the Drury-Arveson space $H^2_d:=\mathcal H(\mathcal S_d)$ on $\mathbb B_d$. In this paper, our investigation centers around the following isometric isomorphism or algebraic isomorphism problem for the multiplier algebras arising from CNP Dirichlet series kernels.\medskip

\textbf{Problem.} \textit{Given CNP Dirichlet series kernels $K_{\mf b,\mf n}$ and $K_{\mf c, \mf m}$, under what conditions are the associated multiplier algebras isometrically isomorphic or algebraically isomorphic?}\medskip

We approach this problem via the general theory of CNP spaces, by which it is known that the multiplier algebra associated with $K_{\mf b, \mf n}$ is canonically isometrically isomorphic to $\mathcal M_{V}$ (see \cite[Theorem~C]{Sa-Sh}) which is the restriction of the multiplier algebra $\mathcal M_d$ of the Drury–Arveson space to a subvariety $V\subseteq \mathbb B_d$, known as the {\it smallest multiplier variety}. More precisely, the smallest multiplier variety $V$ associated with $K_{\mf b, \mf n}$ is given by
\beqn
V = \{\mf z \in \mathbb B_d : g(\mf z)=0\ \text{for all}\ g\in \mathcal M_d\ \text{such that}\ g|_{f_{\mf b, \mf n}(\mathbb H_\rho)}\equiv 0\},
\eeqn
where $f_{\mf b, \mf n}$ is as in ~\eqref{f-at-s-intro}. In a sense, $V$ is the closure of $f_{\mf b, \mf n}(\mathbb H_\rho)$ with respect to the topology generated by multipliers in $\mathcal M_d$.
Thus, an explicit determination of the variety $V$ is crucial for understanding the structure of the corresponding multiplier algebra. It turns out that the geometry of the variety depends intricately on the  weight data $\mf b$ and the arithmetic structure of the frequency data $\mf n$. 
For example, in ~\cite{MS}, it was shown that if $\sum_{j=1}^d b_j=1$ then the logarithms $\{\ln n_j: j=1,\dots,d\}$ are linearly independent over the rational numbers $\mathbb Q$ if and only if $V$ is the full ball $\mathbb B_d$. Therefore, in such a case, the multiplier algebra $\mathcal M(\mathcal H(K_{\mf b, \mf n}))$ is isometrically isomorphic to $\mathcal M_d$. We extend this result to a broader class of kernels, removing the hypothesis $\sum_{j=1}^n b_j=1$,
and provide a new and short proof of the forward implication using only the uniqueness theorem for Dirichlet series (see Theorem~\ref{ind-case-variety}). 

 One of the main contributions of this article is the explicit determination of the multiplier variety in the case where the logarithms $\{\ln n_j: j=1,\dots,d\}$ are linearly dependent over $\mathbb Q$. In this setting, the associated multiplier variety is a proper subvariety of the unit ball $\mathbb B_d$, and we express it as the common zero set of a collection of polynomials determined by the multiplicative relations among the $n_j$'s. To describe the result more precisely, let us define 
 \[
 \begin{array}{rl}
 &\mathcal J(\mf n):= \{ J\subseteq \{1, 2, \ldots, d\}: \{\ln n_j: j \in J\} \text{ is linearly dependent over}\ \mathbb Q \text{ and all of its} \\  & \hspace{3cm} \text{ proper subsets are linearly independent over}\ \mathbb Q\}.
 \end{array}
 \]
Even when $d = \infty,$ each $J \in \mathcal J(\mf n)$ is a finite set. In Proposition~\ref{uniq-decom}, it is shown that corresponding to each $J\in \mathcal J(\mf n)$ there exist a unique partition $\{J_1, J_2\}$ of $J$ and a unique $|J|$-tuple of natural numbers $(\beta_j)_{j\in J}$ with $\text{gcd}(\beta_j)_{j\in J}=1$ such that
\begin{align}\label{arithmetic relation}
 \prod\limits_{i\in J_1}n_i^{\beta_i}=\prod\limits_{i\in J_2}n_i^{\beta_i}.
\end{align}
With these notations, the description of the multiplier varieties is as follows. 
\medskip

\noindent
\textbf{Theorem A.}(Theorem~\ref{variety-gen} below)
\textit{Let $d \in \ntwo \cup \{\infty\}$, and let $(\mf b = (b_j)_{j=1}^d, \mf n= (n_j)_{j=1}^d)$ be a pair of weight and frequency data of length $d$ such that 
$\{\ln n_j: j=1,\dots,d\}$ is linearly dependent over $\mathbb Q$. 
Then, the multiplier variety $V$ associated with $K_{\mf b, \mf n}$ is given by
\beqn
V = \mathbb B_d \cap \bigcap_{J \in \mathcal J(\mf n)} Z(q_{_J}),
 \eeqn
where for each $J\in \mathcal J(\mf n)$ the polynomial
 \begin{equation*}
 q_{_J} (\mf z) = \prod\limits_{i \in J_1}
 b_{i}^{\frac{\beta_i}{2}}\prod\limits_{i \in J_2}z_i^{\beta_i}-\prod\limits_{i \in J_2}b_{i}^{\frac{\beta_i}{2}}\prod\limits_{i \in J_1}z_i^{\beta_i},\quad \mf z = (z_i)_{i=1}^d \in \mathbb C^d,
\end{equation*}
is determined by the associated unique partition $\{J_1, J_2\}$ of $J$ and the unique $|J|$-tuple of integers $(\beta_j)_{j\in J}$ with $\textup{gcd}(\beta_j)_{j\in J}=1$, as in \eqref{arithmetic relation}, and
$Z(q_{_J})$ denotes the zero set of the polynomial $q_{_J}$.}
\medskip

Several examples demonstrate that the above theorem yields a concrete and computable description of the smallest multiplier variety associated with $K_{\mf b, \mf n}$ in the rationally log-dependent case. It is natural to expect that the explicit description of the multiplier varieties will represent a significant step forward in the classification program for multiplier algebras in the setting of CNP Dirichlet series kernels. We illustrate this through several results addressing the (isometric) isomorphism problem for multiplier algebras, as described below.

As a step toward answering the above problem, we first show that if $(\mf b, \mf n)$ and $(\mf c, \mf m)$ are pairs of weight and frequency data 
of different lengths, then the corresponding multiplier algebras associated with the CNP kernels $K_{\mf b, \mf n}$ and $K_{\mf c, \mf m}$ are never isometrically isomorphic (see Proposition~\ref{never-isom-iso} below). This follows from a comparison of the affine dimensions and codimensions of their associated multiplier varieties. Consequently, in addressing the isometric isomorphism problem, it suffices to restrict our attention to CNP kernels with the length of the frequency data being the same.
Now observe that if $V$ is the multiplier variety associated with a CNP kernel $K_{\mf b, \mf n}$ then the composition map
\beq\label{cbn}
C_{f_{\mf b, \mf n}}: \mathcal M_V\to \mathcal M(\mathcal H(K_{\mf b, \mf n})),\quad  g\mapsto g\circ f_{\mf b, \mf n},
\eeq
is an isometric isomorphism (see, \cite[Theorem~C]{Sa-Sh} \& \cite[Section 5]{MS}), where $f_{\mf b, \mf n}$ is as in ~\eqref{f-at-s-intro}. Consequently, for two CNP kernels, a natural route to isometric isomorphism of their multiplier algebras is via the equality of their associated multiplier varieties. This leads to the following question:\medskip 

\textbf{Question 1.} \textit{Under what conditions are the multiplier varieties associated with two CNP kernels equal?}\medskip

To address this question, we introduce a notion of similarity on the collection of all pairs of weight and frequency data.
\begin{definition}
Let $d \in \ntwo \cup \{\infty\}$. Let $(\mf b=(b_i)_{i=1}^d, \mf n)$ and $(\mf c=(c_i)_{i=1}^d, \mf m)$ be two pairs of weight and frequency data of length $d$.
We say that $(\mf b, \mf n)$ and $(\mf c, \mf m)$ admit a {\it similar pattern} if:
\begin{itemize}
\item[(i)] $\mathcal J(\mf n) = \mathcal J(\mf m)$, and
\item[(ii)] for each $J \in \mathcal J(\mf n)$, the unique partition $\{J_1, J_2\}$ and the unique $|J|$-tuple $(\beta_i)_{i \in J}$ associated with $J$ remain the same whether $J$ is viewed as a subset of $\mathcal J(\mf n)$ or $\mathcal J(\mf m)$, and the following identity holds:
\[
\prod\limits_{i \in J_1}
 		c_i^{\beta_i}\prod\limits_{i \in J_2}  
   b_{i}^{\beta_i}=\prod\limits_{i \in J_2}c_i^{\beta_i}\prod\limits_{i \in J_1}b_i^{\beta_i}.
\]
\end{itemize}
\end{definition}
The item (i) in the above definition reflects that $\mf n$ and $\mf m$ share the same arithmetic structure, while item (ii) can be interpreted as a compatibility condition between the weight data $\mf b$ and $\mf c$. This notion of a similar pattern induces an equivalence relation on the collection of pairs of weight and frequency data.
Building on the explicit description of multiplier varieties, we establish that the similarity of weight and frequency data is necessary and sufficient for the corresponding multiplier varieties to coincide. This resolves Question 1 and leads to the following result.\medskip

\noindent
\textbf{Theorem B.}(Proposition~\ref{equality of varieties} and Theorem~\ref{var-equal-set} below)
\textit{ Let $d \in \ntwo \cup \{\infty\}$. Let $(\mf b, \mf n)$ and $(\mf c$, $\mf m)$ be two pairs of weight and frequency data of length $d$. 
Suppose that $V_1$ and $V_2$ are the multiplier varieties associated with $K_{\mf b, \mf n}$ and $K_{\mf c, \mf m}$, respectively.
Then the following statements are equivalent.}

\textup{(i)} \textit{The pairs $(\mf b, \mf n)$ and $(\mf c, \mf m)$ admit a similar pattern.}

\textup{(ii)} $V_1=V_2$.

\textup{(iii)} \textit{The identity map $\mathrm{Id}: \mathcal M_{V_1}\to \mathcal M_{V_2}$ is an isometric isomorphism.}

\textup{(iv)} \textit{There exists an isometric isomorphism $\Phi: \mathcal M (\mathcal H(K_{\mf b, \mf n}))\to \mathcal M(\mathcal H(K_{\mf c, \mf m}))$ such that 
\[C_{f_{\mf c, \mf m}}^{-1}\circ \Phi\circ C_{f_{\mf b, \mf n}}= \mathrm{Id}: \mathcal M_{V_1}\to \mathcal M_{V_2},\] 
where $C_{f_{\mf b, \mf n}}$ is as in ~\eqref{cbn}. }
\medskip 

Item (i) in the above theorem is readily verifiable, which makes the result concrete and yields a range of explicit examples (see Example~\ref{iso-isom-example}) of isometrically isomorphic multiplier algebras. To the best of our knowledge, no alternative approach currently yields such isometric isomorphisms. These examples further demonstrate the subtle and intrinsic dependence of multiplier algebras on the underlying weight and frequency data.

Finally, we examine a distinguished class of CNP Dirichlet series kernels and resolve a question posed by McCarthy and Shalit (\cite{MS}), showing that the answer is negative. This class is defined by specific arithmetic structure of the frequency data. 
Let $d \in \ntwo$, and let $(\mf b, \mf n)$ be a pair of weight and frequency data of length $d$.
Suppose that there exists a subtuple $(n_{k_i})_{i=1}^{d'}$ of the frequency data $\mf n=(n_j)_{j=1}^d$ such that the set $\{\ln n_{k_i}: i=1,\dots,d'\}$ is linearly independent over $\mathbb Q$, and each $n_j$ is a product of non-negative integer powers of the $n_{k_i}$'s, that is,
\beqn
n_j = n_{k_1}^{\alpha_1}n_{k_2}^{\alpha_2} \cdots n_{k_{d'}}^{\alpha_{d'}},
\eeqn
for some tuple $(\alpha_i)_{i=1}^{d'}$ of non-negative integers. In this case, we say that $(n_{k_i})_{i=1}^{d'}$ is the generating subtuple for $\mf n$. 
In this context, the structure of the associated multiplier algebras are already well understood in the extreme cases: 
\begin{itemize}
\item For $d'=1$, it follows from \cite[Corollary~7.4]{DHS} that $\mathcal M(\mathcal H(K_{\mf b, \mf n})) \cong \mathcal M_1$; \item for $d'=d$, the set $\{\ln n_j: j=1\dots, d\}$ is linearly independent over $\mathbb Q$, and the associated multiplier variety is the full ball $\mathbb B_d$, which yields $\mathcal M(\mathcal H(K_{\mf b, \mf n})) \cong \mathcal M_d$.
\end{itemize}
This led McCarthy and Shalit to ask:\medskip


\noindent
\textbf{Question 2.} \textit{Is $\mathcal M(\mathcal H(K_{\mf b, \mf n}))$ isomorphic to $\mathcal M_{d'}$?}\medskip

Unlike in the cases $d'=1$ or $d$, the situation changes drastically. We show that the answer is negative. In fact, there exists a large class of CNP kernels for which the multiplier algebra $\mathcal M(\mathcal H(K_{\mf b, \mf n}))$ is not isomorphic to $\mathcal M_{d'}$. 
Moreover, for this class of CNP kernels, we obtain a necessary and sufficient condition--formulated in terms of their weight and frequency data--for their multiplier algebras to be either algebraically isomorphic or isometrically isomorphic, thereby providing a complete solution to the isomorphism problem.  Remarkably, in this setting, algebraic isomorphism implies isometric isomorphism, and such an isomorphism occurs only when the associated multiplier varieties are equal up to a permutation automorphism. To express our result with precision, let $\sigma$ be a permutation of the index set $\{1,\dots, d\}$, and for any $d$-tuple $\mf b=(b_i)_{i=1}^d$, we define $\sigma(\mf b):=(b_{\sigma(i)})_{i=1}^d$.\medskip



\noindent
\textbf{Theorem C.}(Proposition~\ref{negative} and Theorem~\ref{isometric isomorphism} below) \textit{Let $3\Le d\in \mathbb N$, and let $(\mf b, \mf n)$ and $(\mf c, \mf m)$ be two pairs of weight and frequency data of length $d$.
Suppose that $\{\ln n_{i}: i=1,\dots,d-1\}$ and $\{ \ln m_i: i=1,\dots, d-1\}$ are linearly independent over $\mathbb Q$, and there exist non-empty subsets $I_1$, $I_2$ of $\{1,\ldots,d-1\}$ and tuples $(\alpha_i)_{i\in I_1}$, $(\beta_i)_{i\in I_2}$ of natural numbers such that
\[
n_d=\prod_{i\in I_1}n_i^{\alpha_i}\ \text{ and }\ m_d=\prod_{i\in I_2}m_i^{\beta_i}.
\]
Then  $\mathcal M(\mathcal H(K_{\mf b, \mf n}))$ is not isomorphic to $\mathcal M_{d-1}$. Moreover, the following are equivalent:}
\begin{itemize}
\item[(a)] \textit{$\mathcal M(\mathcal H(K_{\mf b, \mf n}))$ and $\mathcal M(\mathcal H(K_{\mf c, \mf m}))$ are isomorphic.}
\item[(b)] \textit{The pair $(\sigma(\mf b), \sigma(\mf n))$ and $(\mf c, \mf m)$ admit a similar pattern for some permutation $\sigma$ of the index set $\{1,\dots,d\}$. }
\item[(c)] \textit{$\mathcal M(\mathcal H(K_{\mf b, \mf n}))$ and $\mathcal M(\mathcal H(K_{\mf c, \mf m}))$ are isometrically isomorphic. }
\end{itemize}
\medskip
Since $K_{\sigma(\mf b),\sigma(\mf n)}= K_{\mf b, \mf n}$, part (b) of Theorem C establishes a remarkable rigidity phenomenon: the structure of the multiplier algebra uniquely determines the kernel, up to the natural equivalence induced by the associated weight and frequency data exhibiting the same pattern. The strength of Theorem C lies in item (b), which significantly simplifies the problem by reducing it to the arithmetic structure of the frequency and weight data, a condition that can be verified directly. The proof of Theorem C crucially uses the description of multiplier varieties established in Theorem A. We show that the multiplier variety associated with $K_{\mf b, \mf n}$ is not biholomorphic to the ball $\mathbb B_{d-1}$, which proves that $\mathcal M(\mathcal H(K_{\mf b, \mf n}))$ is not isomorphic to $\mathcal M_{d-1}$. For the moreover part, we establish that the multiplier varieties of $K_{\mf b, \mf n}$ and $K_{\mf c, \mf m}$ are biholomorphic if and only if there exists a permutation automorphism of $\mathbb B_d$ which maps one variety onto another.

The results described above are not limited to CNP Dirichlet series kernels, as they extend to classical CNP kernels as well. We establish a correspondence between classical CNP kernels on a domain and CNP Dirichlet series kernels that preserve multiplier varieties. Consequently, the associated multiplier algebras are isometrically isomorphic under this correspondence. This provides a bridge between CNP kernels and CNP Dirichlet series kernels. As a result, one can transfer results from the Dirichlet series setting to the classical setting, yielding new conclusions even in the classical context. See Theorem~\ref{cnp classical} for an example. These aspects are elaborated in the final section.  

The article is organized as follows. In the next section, we briefly recall the Hilbert space of Dirichlet series and review some basic properties. In Section~\ref{3}, we characterize all CNP Dirichlet series kernels and provide illustrative examples. An explicit description of the multiplier varieties associated with CNP Dirichlet series kernels is obtained in Section~\ref{4}. The problem of isometric isomorphism of multiplier algebras is addressed in Sections~\ref{5}. In Section~\ref{6}, we examine a distinguished class of CNP kernels to resolve a question of McCarthy and Shalit and obtain a complete classification of multiplier algebras--up to algebraic and isometric isomorphism--for this class of CNP kernels.

\section{Hilbert space of Dirichlet series}

 A {\it Dirichlet series} is a series of the form
$$f(s) = \sum_{n=1}^{\infty} a_{n} n^{-s},$$
where $a_n$ are complex numbers and $s$ is a complex variable. 	
A primary example is the Riemann zeta function $\zeta(s)=\sum_{n=1}^\infty n^{-s}$.
If $f$ is convergent at $s=s_{0}$, then it converges uniformly throughout the angular region 
\[\{s \in \mathbb C : |\arg(s-s_0)| < \frac{\pi}{2} -\delta\}\] for every real number $0< \delta < \frac{\pi}{2}$. Consequently, the series $f$ defines a holomorphic function on the right half-plane $\mathbb H_{\text{Re}(s_0)}=\{s\in\mathbb C: \text{Re}(s) > \text{Re}(s_0) \}$. For further details, see the excellent exposition on the theory of Dirichlet series \cite[Chapter~IX]{Ti}.
The abscissa of absolute convergence $\sigma_a(f)$ of a Dirichlet series $f$ is given by 
\beq
\label{sigma-a}
\sigma_a(f) : = \inf\left\{\text{Re}(s) : \sum_{n=1}^\infty |a_n| n^{-\text{Re}(s)}~\text{is convergent} \right\}.
\eeq

The following result asserts that a Dirichlet series is uniquely determined by its values on any right half-plane where it converges. 

\begin{proposition}$($Uniqueness of the Dirichlet series$)$.
	\label{uniq-dir}
For $\rho \in \mathbb R,$ suppose that the Dirichlet series $f(s) = \sum_{n=1}^\infty a_nn^{-s}$ converges on $\mathbb H_\rho$. If $f = 0$ on $\mathbb H_\rho$, then $a_n = 0$ for every integer $n \Ge 1.$
\end{proposition}

We will use this uniqueness property of Dirichlet series repeatedly throughout the article. The next proposition describes the product of two Dirichlet series (see \cite[Theorem~4.3.1 and discussion prior to Theorem~4.3.4]{QQ}). 

\begin{proposition}\label{prod-series}
Let $f(s) = \sum_{n=1}^\infty a_n n^{-s}$ and $g(s) = \sum_{n=1}^\infty b_n n^{-s}$ be two convergent Dirichlet series on $\mathbb H_\rho$, for some $\rho\in\mathbb R$. If $g$ converges absolutely on $\mathbb H_\rho,$ then the product is given by 
	\beqn
	f(s)g(s) = \sum_{n=1}^\infty \Big(\sum_{m | n} a_m b_{\frac{n}{m}}\Big) n^{-s},
	\eeqn
and it converges on $\mathbb H_\rho.$
\end{proposition}

Let $\mathcal D$ denote
the set of all functions representable by a convergent Dirichlet series
in some right half-plane. The following proposition determines a necessary and sufficient condition under which the inverse of a Dirichlet series can also be expressed as a Dirichlet series (see, \cite[Section~6.2]{Mc-D}).
\begin{proposition}\label{reciprocal-Dir}
Let $f  \in \mathcal D$. Then, $\frac{1}{f} \in \mathcal D$ if and only if $\lim\limits_{\text{Re}(s) \rar \infty}f(s) \neq 0$.
\end{proposition}

Let $\Omega$ be a nonempty set, and let $K:\Omega \times \Omega \to \mathbb C$ be a \emph{positive semi-definite kernel} on $\Omega$, that is, 
for any finite set $\{x_1,\ldots,x_n\}\subseteq \Omega$, the matrix  
\[\begin{bmatrix}
K(x_i,x_j)
\end{bmatrix}_{i,j=1}^n\] 
is positive semi-definite.
By Moore's theorem, such a kernel gives rise to a unique reproducing kernel Hilbert space $\clh(K)$ of scalar-valued functions on $\Omega$ associated with $K$ (see \cite{PaulRaghu}). In the context of Dirichlet series, let $\mathbf w = \{w_{n}\}_{n = 1}^{\infty}$ be a sequence of non-negative real numbers such that
\[\rho_\mf w: = \sigma_a(\sum\limits_{n=1}^\infty w_n n^{-s}) < \infty.\]
Then the function
\beqn
K_{\mf w}(s, u): =  \sum_{n = 1}^{\infty} w_{n} n^{-s-\overline{u}}, \quad s, u \in \mathbb H_{\frac{\rho_{\mf w}}{2}},
\eeqn
defines a positive semi-definite kernel on the right half-plane $\mathbb H_{\frac{\rho_{\mf w}}{2}}$ (see \cite[Lemma~20]{MS}). 
The reproducing kernel Hilbert space $\clh(K_{\mf w})$ associated with $K_\mf w$ is given by
 \beqn 
 \mathcal H(K_{\mf w})  = \Bigg\{\sum_{n = 1}^{\infty} a_n n^{-s}   : \|f\|^2_{\mf w} := \sum_{n = 1}^{\infty} \frac{|a_n|^{2}}{w_{n}} < \infty\Bigg\},
 \eeqn
where we adopt the convention that $n^{-s}$ is omitted from $\clh(K_\mf w)$ whenever $w_n = 0$.
Furthermore, for every $s \in \mathbb H_{\frac{\rho_{\mf w}}{2}},$ $K_\mf w(\cdot, s) \in \clh(K_{\mf w})$ and 
\beqn
\inp{f}{K_\mf w(\cdot, s)} = f(s), \quad f \in \clh(K_{\mf w}).
\eeqn
Note that $K_\mf w$ converges absolutely on $\mathbb H_{\frac{\rho_{\mf w}}{2}} \times \mathbb H_{\frac{\rho_{\mf w}}{2}}.$ 
We say $K_{\mathbf w}$ is {\it normalized} if $w_1 = 1$, or equivalently, if 
\[K_\mf w(\cdot, \infty):= \lim_{\sigma \rar \infty} K_\mf w(\cdot, \sigma)= 1.\] 
For instance, if $\mf 1$ denotes the constant sequence with value $1$, then $\rho_\mf 1 =1$ and the corresponding space  
\beqn
 \clh(K_\mf 1) = \Bigg\{\sum_{n = 1}^{\infty} a_n n^{-s}   :  \sum_{n = 1}^{\infty} |a_n|^{2} <  \infty\Bigg\}
\eeqn
is known as the {\it Hardy space of Dirichlet series} associated with the reproducing kernel 
\begin{equation}
\label{zeta-kernel}
    K_\mf 1(s, u) = \zeta(s+\overline{u}) = \sum_{n=1}^\infty n^{-s-\overline{u}}, \quad s, u \in \mathbb H_{1/2}.
\end{equation}

Let $K$ be a kernel on $\Omega$ and $\clh(K)$ be the associated reproducing kernel Hilbert space. A function $\varphi:\Omega \to \C$ is called a \emph{multiplier} on $\clh(K)$ if $\varphi f\in \clh(K)$ for all $f\in \clh(K)$. By an application of closed graph theorem, for each multiplier $\varphi$ on $\clh(K)$, the multiplication operator 
\[M_\varphi: \clh(K) \to \clh(K),\ f\mapsto \varphi f,
\] 
is bounded on $\clh(K)$. The collection of all multipliers on $\clh(K)$ forms a unital Banach algebra, denoted by $\clm(\clh(K))$, and is equipped with the norm $\|\varphi\|:= \|M_\varphi\|,$ for $\varphi \in \clm(\clh(K)).$

\section{CNP Dirichlet series kernels}\label{3}
Let $K$ be a kernel on a non-empty set $\Omega$. The matrix-valued Nevanlinna–Pick interpolation problem asks the following: given a finite set of points $\{x_1, \ldots, x_n\}\subseteq \Omega$ and matrices $W_1,\ldots, W_n \in M_{p,q}(\mathbb C)$ (the space of all $p\times q$ complex matrices), does there exist a multiplier $\varphi$ in 
\[\clm(\clh(K)\otimes \C^q,\clh(K)\otimes \C^p):=\{f:\Omega \to M_{p,q}(\mathbb C):f g\in \clh(K)\otimes \C^p \text{ for all }g\in\clh(K)\otimes \C^q\}\] such that
\[\varphi(x_i)=W_i \ (i=1,2,\ldots,n)\ \text{ and }\ \|M_\varphi\|\Le 1.\]
The contractivity requirement $\|M_{\varphi}\|\Le 1$ imposes the necessary condition that the Pick matrix
\[\begin{bmatrix}
(I-W_i W^*_j)K(x_i,x_j)
\end{bmatrix}_{i,j=1}^{n}\]
be positive semi-definite (see \cite[Theorem 5.8]{AM}). We say $K$ has the $M_{p,q}$-Nevanlinna-Pick property if this necessary condition is also sufficient for the existence of such a multiplier. A kernel $K$ is called a \emph{complete Nevanlinna-Pick kernel} (abbreviated CNP kernel) if $K$ has the $M_{p,q}$-Nevanlinna-Pick property for all $p,q\in \mathbb N$. In this case, the associated reproducing kernel Hilbert space $\clh(K)$ is called a CNP space. Prototype examples of CNP spaces include: 
\begin{itemize}
\item 
the Hardy space $H^2(\D)$ over the unit disc with Szeg$\ddot{o}$ kernel $\cls_1(z,w)=\frac{1}{1-z\bar{w}}$,
\end{itemize}
and
\begin{itemize}
\item 
the Drury-Arveson space $H^2_n$ over the unit ball $\mathbb B_n$ in $\mathbb C^n$ with kernel 
\begin{align*}
    \mathcal S_n(\z,\w)=\frac{1}{1-\langle \z,\w\rangle}_{\C^n},\quad \mf z, \mf w \in \mathbb B_n.
\end{align*} 
\end{itemize}
An excellent reference for the theory of the Nevanlinna–Pick interpolation problem and complete Nevanlinna–Pick spaces is the book by Agler and McCarthy \cite{AM}.
Building on the general framework, McCarthy and Shalit investigated the case of Dirichlet series kernels and established the following characterization of CNP Dirichlet series kernels (see \cite[Theorem~26]{MS}).
\begin{proposition}
	\label{cnp-char}
Let $\mathbf w= \{w_n\}_{n=1}^{\infty}$ be a sequence of non-negative real numbers, and let 
\[
K_{\mf w}(s, u): =  \sum_{n = 1}^{\infty} w_{n} n^{-s-\overline{u}}\]
be the associated Dirichlet series kernel defined on some right half-plane. 
Suppose that $w_1 \neq 0$ and the reciprocal Dirichlet series is given by   
	\beqn
	\frac{1}{\sum_{n = 1}^\infty w_nn^{-s}}=\sum_{n = 1}^\infty c_nn^{-s},
	\eeqn
    on some right half-plane.
	Then, $K_\mf w$ is a CNP kernel if and only if $c_n \Le 0$ for all $n \Ge 2.$
\end{proposition}

Based on the above characterization, we next obtain an alternative description of CNP Dirichlet series kernels. This leads to a cohesive and general formulation of such kernels, and plays a critical role in identifying the smallest multiplier varieties, which we consider in the next section.

\begin{theorem}\label{all-cnp-char}
Let $d \in \mathbb{N} \cup \{\infty\}$, and let $(\mf b = (b_j)_{j=1}^d, \mf n= (n_j)_{j=1}^d)$ be a pair of weight and frequency data of length $d$ such that $\sigma_a(\sum_{j=1}^d b_j n_j^{-s}) < \infty$.
Then, there exists $\rho \in \mathbb R$ such that 
the kernel $K_{\mf b, \mf n}$ is a CNP kernel, where 
\beq
\label{kappa-b}
K_{\mf b, \mf n}(s, u) = \frac{1}{1 - \sum_{j=1}^d b_j n_j^{-s-\overline{u}}}, \quad s, u \in \mathbb H_\rho.
\eeq

Conversely, if $K$ is a non-constant normalized CNP kernel defined on some right half-plane, then $K$ is of the form $K_{\mathbf{b},\mathbf n}$, as in \eqref{kappa-b}, for some pair $(\mf b,\mf n)$ of weight and frequency data of length $d\in \mathbb N \cup \{\infty\}$ satisfying $\sigma_a(\sum_{j=1}^d b_j n_j^{-s}) < \infty$.
 \end{theorem}
\begin{proof}Since
$$
\lim_{\sigma \rar \infty} \sum_{j=1}^d b_j n_j^{-\sigma} = 0,
$$
there exists a real number $\rho$ such that $\sum_{j=1}^d b_j n_j^{-2\rho} < 1$, and hence $|\sum_{j=1}^d b_j n_j^{-s - \overline{u}}| < 1$ for all $s, u \in \mathbb H_\rho$. With this choice of $\rho$, since  $(s,u) \mapsto \sum_{j=1}^d b_jn_j^{-s-\overline{u}}$ is a positive semi-definite kernel on $\mathbb H_{\rho}$, 
    by the Schur product theorem, and the fact that positive semi-definiteness is preserved under limits, it follows that $K_{\mf b, \mf n}$ is a positive semi-definite kernel on $\mathbb H_\rho.$ Therefore, from the form of $K_{\mf b, \mf n}$, since 
$$
\frac{1}{K_{\mf b, \mf n}} = 1 - \sum_{j=1}^d b_j n_j^{-s-\overline{u}}, 
$$
by Proposition~\ref{cnp-char}, $K_{\mf b, \mf n}$ is a CNP kernel. 

To prove the converse, suppose that $K$ is a non-constant normalized CNP kernel defined on a right half-plane $\mathbb H_\rho$. Then, $K$ can not be defined on the whole complex plane $\mathbb C$ (see \cite[Theorem~28]{MS}), which ensures that $\rho \in \mathbb R.$ In this case, $K$ must be of the form
    \beqn
    K(s,u)=1+\sum_{n=2}^\infty a_nn^{-s-\overline{u}}, \quad s, u \in \mathbb H_\rho,
    \eeqn
where $a_n > 0$ for infinitely many $n$. By Proposition~\ref{reciprocal-Dir}, $(1+\sum_{n=2}^\infty a_nn^{-s})^{-1}$ is in $\mathcal D$, which implies the existence of $d \in \mathbb N \cup \{\infty\}$, $\rho' \in \mathbb R$, and a $d$-tuple $(x_j)_{j=1}^{d}$ of elements in $\mathbb C\setminus\{0\}$, along with a $d$-tuple $\mf n = (n_j)_{j=1}^{d}$ of distinct elements in $\ntwo$, such that
	\beq
 \label{K-inverse}
	K^{-1}(s,u) = 1 + \sum_{j=1}^{d} x_jn_j^{-s-\overline{u}}, \quad  s, u \in \mathbb H_{\rho'}.
	\eeq
Since $K$ is CNP, Proposition~\ref{cnp-char} implies $x_n < 0$ for all $n$. Define $b_n= -x_n>0$, and set $\mf b := (b_j)_{j =1}^{d}$. Then
\beqn
K(s,u) = \frac{1}{1-\sum_{j = 1}^{d} b_j n_j^{-s-\overline{u}}}  = K_{\mf b, \mf n}(s,u), \quad s, u \in \mathbb H_\rho.  
\eeqn
Also, by \eqref{K-inverse}, $\sum_{j = 1}^{d} b_j n_j^{-s}$ converges on $\mathbb H_{\frac{\rho'}{2}}$, which implies that $\sigma_a(\sum_{j = 1}^{d} b_j n_j^{-s})< \infty$. This completes the proof.
\end{proof}

By the above theorem,  a CNP Dirichlet series kernel is completely determined by a pair $(\mf b, \mf n)$ that consists of weight data and frequency data of the same length. 
Before proceeding further, we present a class of examples illustrating Theorem~\ref{all-cnp-char}.

\begin{example} 
\label{kappa-m}
(i) Let $d$ be a positive integer. Let $\mf b = (b_j)_{j=1}^d$ be a $d$-tuple of positive real numbers and $\mf n= (n_j)_{j=1}^d$ be a $d$-tuple of distinct elements in $\ntwo$. Consider a Dirichlet polynomial
\[f(s) = \sum_{j=1}^d b_jn_j^{-s}.\] 
Since $f$ is entire, we have $\sigma_\mf a(f) =-\infty$. Also, as $\lim\limits_{\sigma \rar -\infty}f(\sigma) = \infty$ and $\lim\limits_{\sigma \rar \infty}f(\sigma)  =0$, there exists a unique real number $\alpha$ such that $f(\alpha)=1.$ 
Hence, by Theorem~\ref{all-cnp-char}, $K_{\mf b, \mf n}$ is a CNP kernel on $\mathbb H_{\frac{\alpha}{2}}.$

(ii) For a real number $r>1$, let 
\[c: = \sum_{n=2}^\infty\frac{1}{n(\log n)^2} < \infty,\]
and define $\mf b = (b_j)_{j=1}^\infty$ by 
\beqn
b_j = \frac{1}{(c+1)(j+1)^r(\log (j+1))^2}, \quad j \Ge 1,
\eeqn
and set $\mf n = \ntwo.$ Then, a simple computation shows that the series \[\sum_{j =1}^\infty b_j (j+1)^{-\sigma}\] converges on $[1-r, \infty)$ and takes the value $\frac{c}{c+1} \in (0, 1)$ at $\sigma = 1-r.$ Moreover, by the integral test for convergence, the series diverges 
for $\sigma < 1-r$, implying that $\sigma_{\mf a}\big(\sum_{j=1}^\infty b_j (j+1)^{-s}\big)=1-r$. Therefore, by Theorem~\ref{all-cnp-char}, $K_{\mf b, \mf n}$ is a CNP kernel on $\mathbb H_{\frac{1-r}{2}}.$
        
(iii) For an integer $m \Ge 1,$ let $\rho_m \in (1, \infty)$ be the unique real number satisfying $\zeta(\rho_m)^m = 2$, where $\zeta$ denotes the Riemann zeta function. Consider the kernel
    \beq\label{Z_m_ker}
K_m(s, u): = \frac{1}{2- \zeta(s+\overline{u})^m}, \quad s, u \in \mathbb H_{\frac{\rho_m}{2}}.
    \eeq
 For every integer $n \Ge 2,$ let $d_m(n)$ denote the number of ways writing $n$ as a product of $m$ factors, where the order matters. Then, by \cite[Equation~(1.2.2)]{Ti}, we have 
\beqn
 K_m(s, u) = \frac{1}{1- \sum_{n=2}^\infty d_m(n)n^{-s-\overline{u}}},
\eeqn
and hence, by Theorem~\ref{all-cnp-char}, $K_m$ is a CNP Dirichlet series kernel on $\mathbb H_{\frac{\rho_m}{2}}.$ For the case $m=1$, the CNP kernel $K_1$ has been studied in detail in \cite[Section~3]{Mc-1}.
\end{example}
As an application of Theorem~\ref{all-cnp-char}, the following result characterizes a class of CNP Dirichlet series kernels that admit the kernel of the Hardy space of Dirichlet series as a CNP factor.


\begin{corollary}\label{zeta-factor}
Let $K_\mf 1$ be the kernel of the Hardy space of Dirichlet series as defined in \eqref{zeta-kernel}. Let $K_{\mf b, \mf n}$ be the CNP kernel, as in \eqref{kappa-b}, associated with an infinite tuple $\mf b = (b_j)_{j=1}^\infty$ of positive real numbers and $\mf n =(2,3,\dots)$. 
Then the quotient $\frac{K_{\mf b, \mf n}}{K_\mf 1}$ is a CNP kernel on $\mathbb H_\frac{\rho}{2}$ for some $\rho > 1$ if and only if for all $n\Ge 2$,
\beq\label{zeta-fac-suff}
	 \sum_{\substack{ m \Ge 2\\ m|n}} b_{m-1} \Ge 1.
 \eeq
In this case, the inclusion $H^\infty(\mathbb H_0)\cap  \mathcal D \subseteq \mathcal M(\clh(K_{\mf b, \mf n}))$ is a complete contraction.
\end{corollary}
\begin{proof}
Suppose that the inequality \eqref{zeta-fac-suff} holds. Then in particular, for every prime number $p$, we have $b_{p-1} \Ge 1$, and hence
\beqn
\sum_{n=2}^\infty b_{n-1} n^{- s} \Ge \sum_{n=1}^\infty p_n^{-s},
\eeqn
where $p_n$ denotes the $n$-th prime number. Since the series of reciprocals of primes diverges, it follows that the abscissa of absolute convergence of $\sum_{n=2}^\infty b_{n-1} n^{- s}$ is at least $1$, and the series diverges for $s=1$. Therefore, by Theorem~\ref{all-cnp-char}, $K_{\mf b, \mf n}$ is a CNP kernel on $\mathbb H_\frac{\rho}{2}$ for some $\rho > 1.$  
Now using Proposition~\ref{prod-series}, one obtains  
\beq
\label{ratio-kernel}
\frac{K_{\mf b, \mf n}(s,u)}{K_\mf 1(s, u)} \overset{\eqref{zeta-kernel}}= \frac{1}{\Big(1- \sum\limits_{n=1}^\infty b_n(n+1)^{-s-\overline{u}}\Big) \zeta(s+\overline{u})} = \frac{1}{1-\sum_{n \Ge 2}c_n n^{-s-\overline{u}}}, \quad s, u \in \mathbb H_{\frac{\rho}{2}},
	\eeq
where the coefficients are given by 
\[c_n = -1+\sum\limits_{\substack{m \Ge 2\\ m|n}} b_{m-1}, \quad n \Ge 2. 
\]
Since the inequality \eqref{zeta-fac-suff} holds, $c_n \Ge 0$ for all $n \Ge 2$. As a result, by Theorem~\ref{all-cnp-char},  the identity \eqref{ratio-kernel} implies $\frac{K_{\mf b, \mf n}}{K_\mf1}$ is a CNP Dirichlet series kernel on $\mathbb H_{\frac{\rho}{2}}$. 

%

Conversely, suppose that $\frac{K_{\mf b, \mf n}}{
K_\mf1}$ is a CNP kernel. Then, by Theorem~\ref{all-cnp-char} and the identity in \eqref{ratio-kernel}, we must have $c_n \Ge 0$ for all $n \Ge 2,$ which implies the required inequality. 

For the final part, recall from Hedenmalm, Lindqvist, and Seip (\cite[Theorem~3.1]{HLS}) that $\mathcal M(\clh(K_\mf1)) = H^\infty(\mathbb H_0)\cap \mathcal D$, and this identification is isometric. Since $\frac{K_{\mf b,\mf n}}{K_\mf1}$ is a positive semi-definite kernel, by ~\cite[Corollary~3.5]{Hartz_1}, 
the desired conclusion follows.
\end{proof}	

We conclude the section by revisiting an example of a CNP kernel that admits the kernel of the Hardy space of Dirichlet series as a CNP factor. 
 \begin{example}[Example~\ref{kappa-m} continued]
For an integer $m \Ge 1$, consider the kernel $K_m$ on $\mathbb H_{\frac{\rho_m}{2}}$ as defined in \eqref{Z_m_ker}. Note that $b_n = d_m(n) \Ge 1$ for all integers $n \Ge 1$. This implies that the condition in \eqref{zeta-fac-suff} is satisfied. Hence, by Corollary~\ref{zeta-factor}, the quotient $\frac{K_m}{K_\mf 1}$ is a CNP kernel and the inclusion $H^\infty(\mathbb H_0)\cap  \mathcal D \subseteq \mathcal M(\clh(K_m))$ holds completely contractively.	
\end{example}

\section{Multiplier varieties associated with CNP Dirichlet series kernels}\label{4}

This section is devoted to determining the multiplier variety associated with CNP kernels of Dirichlet series. Let $d \in \mathbb N \cup \{\infty\}$, and let $(\mf b, \mf n)$ be a pair of weight and frequency data of length $d$. Let the associated CNP kernel $K_{\mf b,\mf n}$ be defined on $\mathbb H_\rho$. 
Then 
\beqn
K_{\mf b, \mf n}(s,u) = \frac{1}{1-\sum_{j=1}^d b_j n_j^{-s-\overline{u}}} = \frac{1}{1-\inp{f_{\mf b, \mf n}(s)}{f_{\mf b, \mf n}(u)}},\quad  s, u \in \mathbb H_\rho,
\eeqn
where the map  $f_{\mf b, \mf n}:\mathbb H_\rho \rar \mathbb B_d$ is defined by 
\beq
\label{f-at-s}
f_{\mf b, \mf n}(s) =(\sqrt{b_j} n_j^{-s})_{j=1}^d, \quad s \in \mathbb H_\rho.
\eeq
Recall that the smallest multiplier variety associated with $K_{\mf b,\mf n}$ is given by 
\beq
\label{def-V}
V = \{\mf z \in \mathbb B_d : g(\mf z)=0\ \text{for all}\ g\in \mathcal M_d\ \text{such that}\ g|_{f_{\mf b, \mf n}(\mathbb H_\rho)}\equiv 0\},
\eeq
where $\mathcal M_d$ is the multiplier algebra of the Drury-Arveson space $H^2_d$ over $\mathbb B_d$. 
Observe that $f_{\mf b, \mf n}(\mathbb H_\rho) \subseteq V \subseteq \mathbb B_d.$ We show below that $0\in V$.
The following notations will be in use: for the set of all non-negative integers $\mathbb Z_{+}$ and $d\in\mathbb N$, $\mathbb Z_{+}^d$ denotes the $d$-fold cartesian product of $\mathbb Z_{+}$; when $d=\infty$, 
\[
\mathbb Z_{+}^d:=\{(n_i)_{i\Ge 1}: n_i\in \mathbb Z_{+}, \text{ and}~~n_i = 0~~ \text{for all}~ i ~\text {sufficiently large}\}.
\]


\begin{remark}
  \label{0-in-V}  
Let $V$ be the multiplier variety associated with $K_{\mf b, \mf n}$. Then $0 \in V.$ Indeed, let $g\in \mathcal M_d$ be such that $g|_{f_{\mf b, \mf n}(\mathbb H_\rho)}\equiv0$, where $f_{\mf b, \mf n}$ is given by \eqref{f-at-s}. Then $g \in H_d^2$, and therefore it admits a power series expansion:  
 	\beq
 	\label{g-power}
 	g(\mf z)=\sum_{\alpha \in \mathbb Z_+^d}c_\alpha \mf z^\alpha, \quad \mf z \in \mathbb B_d. 
 	\eeq
 Since $f_{\mf b, \mf n}(\mathbb H_\rho) \subseteq \mathbb B_d$ and $g|_{f_{\mf b, \mf n}(\mathbb H_\rho)}\equiv0$, evaluating at $f_{\mf b, \mf n}(s)$ we obtain
  \beqn
 g(f_{\mf b, \mf n}(s)) \overset{\eqref{g-power}}= \sum\limits_{\alpha \in \mathbb Z_+^d}c_\alpha \prod\limits_{j=1}^d b_{j}^{\frac{\alpha_j}{2}} \Big(\prod\limits_{j=1}^d n_j^{\alpha_j}\Big)^{-s} = 0, \quad s \in\mathbb H_\rho.
 \eeqn
By the uniqueness of Dirichlet series (see Proposition~\ref{uniq-dir}), the constant term must vanish, that is, $c_0 = 0$. Hence
$g( 0) = c_0 = 0 $, which proves $0\in V$.	
\end{remark}
Let us now analyze the multiplier variety for the simplest case $d =1$. Fix an integer $n \Ge 2,$  and set $\mf b = (b)$ and $\mf n = (n)$. Then the associated kernel is given by
\beqn
K_{\mf b, \mf n}(s, u) = \frac{1}{1- bn^{-s-\overline{u}}},\quad  s, u \in \mathbb H_\rho,
\eeqn
for some $\rho\in\mathbb R$. Note that $\mathcal M_1= H^\infty(\mathbb D)$, as $H^2_1$ is the Hardy space over the unit disc $\mathbb D$. Since
\beqn
f_{\mf b, \mf n}(\mathbb H_\rho) = \{\sqrt{b}n^{-s} : s\in \mathbb H_\rho\},
\eeqn
 any function $g\in H^{\infty}(\mathbb D)$ that vanishes on $\{\sqrt{b}n^{-s} : s\in \mathbb H_\rho\}$ must be identically zero. It then follows that 
$V = \mathbb D.$ Thus, it is essential to consider the case $d>1$. In such a setting, we analyze the frequency data $\mf n=(n_j)_{j=1}^d$ by distinguishing two following scenarios:
\begin{itemize}
\item 
The independent case, where $\{\ln n_j: j=1,\dots, d\}$ is a linearly independent set over $\mathbb Q$. 
\end{itemize}

\begin{itemize}
\item 
The dependent case, where the set $\{\ln n_j: j=1,\dots, d\}$ is linearly dependent over $\mathbb Q$.
\end{itemize}
These two cases are considered separately in the following two subsections.

\subsection{Independent case} In \cite[Theorem~41]{MS}, under the assumption $\sum_{n=1}^d b_n = 1,$ the authors characterized all CNP Dirichlet series kernels for which the associated multiplier variety $V$ coincides with the unit ball $\mathbb B_d.$ However, this result remains valid even for a more general class of Dirichlet series kernels, as given below.

 \begin{theorem}\label{ind-case-variety}
 Let $d \in \ntwo \cup \{\infty\}$, and let $(\mf b, \mf n= (n_j)_{j=1}^d)$ be a pair of weight and frequency data of length $d$. 
 Then the multiplier variety $V$ associated with the kernel $K_{\mf b, \mf n}$ is equal to $\mathbb B_d$ if and only if the set $\{\ln n_j: j=1, \dots , d\}$ is linearly independent over $\mathbb Q$.
\end{theorem}
 
 In \cite[Theorem~41]{MS}, the proof of the converse direction (assuming $\sum_{j=1}^d b_{j} =  1$) relies on the maximum modulus principle on the polydisc and Kronecker's theorem from Ergodic theory. While one can adapt those arguments to prove Theorem~\ref{ind-case-variety}, we present here a shorter and more direct proof of the converse direction, using the uniqueness property of Dirichlet series and the following lemma.

 \begin{lemma}\label{unique-term}	
Let $d \in \ntwo \cup \{\infty\},$ and let $(n_j)_{j=1}^d$ be a $d$-tuple of distinct elements in $\ntwo$. Then the set $\{\ln n_j: j=1,\dots, d\}$ is linearly independent over $\mathbb Q$ if and only if for all $d$-tuple of rational numbers $(\gamma_j)_{j=1}^d, (\beta_j)_{j=1}^d $ with only finitely many non-zero terms,
 	 \beqn
 	 \prod_{j=1}^d n_j^{\gamma_j} = \prod_{j=1}^d n_j^{\beta_j} \implies \gamma_j = \beta_j \text{ for all } j.
 	 \eeqn 
 \end{lemma}
 \begin{proof}
 The proof follows from the fact that, $\{\ln n_j: j=1,\dots,d\}$ is linearly dependent over $\mathbb Q$ if and only if there exists rational numbers 
    $c_1, c_2, \ldots, c_d$, with not all zero and only finitely many are non-zero, such that
 	\beqn
 	\sum_{j=1}^d c_j \ln {n_j} = 0\iff \prod_{j=1}^d n_j^{c_j}=1.
 	\eeqn
 \end{proof}

 \begin{proof}[\textbf{Proof of Theorem \ref{ind-case-variety}}]
 The forward implication can be established by arguments identical to those in \cite[Theorem~41]{MS}, and we leave the details to the reader. 
 For the converse, suppose that the set $\{\ln n_j: j=1,\dots, d\}$ is linearly independent over $\mathbb Q$ and the kernel $K_{\mf b, \mf n}$ is defined on $\mathbb H_\rho$. By the definition of multiplier variety, $V \subseteq \mathbb B_d.$ We show that $\mathbb B_d \subseteq V.$ To this end, let $g\in\mathcal{M}_d$ be such that $g|_{f_{\mf b, \mf n}(\mathbb H_\rho)}\equiv0$. Since $g\in H^2_d$, $g$ admits a power series expansion 
 \[g(\mf z)=\sum_{\alpha \in \mathbb Z_+^d}c_\alpha \mf z^\alpha,\quad \mf z \in \mathbb B_d.\]
Evaluating at $f_{\mf b,\mf n}(s)$, we obtain 
\beqn
 	g(f_{\mf b, \mf n}(s)) = \sum_{\alpha \in \mathbb Z_+^d}c_\alpha \prod_{j=1}^d b_{j}^{\frac{\alpha_j}{2}} \Big(\prod_{j=1}^d n_j^{\alpha_j}\Big)^{-s} = 0,\quad s \in\mathbb H_\rho.
 	\eeqn
By Lemma~\ref{unique-term} and the uniqueness of Dirichlet series (Proposition~\ref{uniq-dir}),  $c_\alpha \prod_{j=1}^d b_{j}^{\frac{\alpha_j}{2}}=0$ for all $\alpha \in \mathbb Z_+^d$. Since $b_j > 0$ for all $j$, this implies $c_\alpha = 0$ for all $\alpha \in \mathbb Z_+^d$. Thus, $g\equiv0$ on $\mathbb B_d$, and hence $\mathbb B_d \subseteq V.$ This completes the proof. 
\end{proof}






\subsection{Dependent case} 
In this subsection, our goal is to determine the multiplier variety of the Dirichlet series kernel $K_{\mf b, \mf n}$, when the set $\{\ln n_j: j=1,\dots,d\}$ is linearly dependent over $\mathbb Q$. 
The proof of the main result in this setting is somewhat lengthy and relies on several auxiliary lemmas. We begin with an important proposition, which will play a crucial role in what follows.
To state it, we recall some notation introduced in the Introduction. 


Let $\mf n = (n_j)_{j=1}^d$ be a $d$-tuple of distinct elements in $\ntwo$ such that the set $\{\ln n_j: j=1, \dots, d\}$ is a linearly dependent set over $\mathbb Q$. Recall that $\mathcal J(\mf n)$ denotes the collection
of all subsets $J \subseteq \{1, 2, \ldots, d\}$ such that the set $\{\ln n_j: j \in J\}$ is linearly dependent over $\mathbb Q$, but every proper subset is linearly independent over $\mathbb Q$. Observe that even when $d = \infty,$ each $J \in \mathcal J(\mf n)$ is necessarily a finite set. 
Corresponding to each $J\in\mathcal J(\mf n)$, the following result provides a unique partition and a unique $|J|$-tuple of integers depending on the relations of $\ln n_i$'s.

\begin{proposition}\label{uniq-decom}
    Fix $d \in \ntwo \cup \{\infty\}$, and let $\mf n = (n_j)_{j=1}^d$ be a $d$-tuple of distinct elements in $\ntwo$ such that $\{\ln n_j: j \in J\}$ is a linearly dependent set over $\mathbb Q$. If $J \subseteq \{1,\ldots,d\}$ belongs to $\mathcal J(\mf n)$,
then there exist a unique tuple $(\beta_j)_{j \in J} \in \mathbb N^{|J|}$ satisfying $\gcd(\beta_j)_{j \in J}=1$, and a unique partition $\{J_1, J_2\}$ of $J$ such that
 \beq \label{uniq-rep}
 		\prod_{i \in J_1}	n_i^{\beta_i} = \prod_{i \in J_2} n_i^{\beta_i}.
 \eeq
\end{proposition}

In Proposition~\ref{uniq-decom}, note that the unique tuple $(\beta_j)_{j \in J} \in \mathbb N^{|J|}$ depends on both $J \in \mathcal J(\mf n)$ and the frequency data $\mf n$.
 We refer to $(\beta_j)_{j \in J}$ and $\{J_1, J_2\}$ as {\it the unique tuple and unique partition associated with $J$}, respectively. Throughout, we reserve the notation $J_1$ and $J_2$ for the unique partition of $J$ satisfying \eqref{uniq-rep}. We also refer to  \eqref{uniq-rep} as {\it the representation associated with $J.$} 
 Using this representation, we define a polynomial $q_{_J}$ associated with $J$ as follows:   
        \beqn
q_{_J}(\mf z) = \prod\limits_{i \in J_1}
 		b_{i}^{\frac{\beta_i}{2}}\prod\limits_{i \in J_2}z_i^{\beta_i}-\prod\limits_{i \in J_2}b_{i}^{\frac{\beta_i}{2}}\prod\limits_{i \in J_1}z_i^{\beta_i},\quad \mf z \in \mathbb C^d.
    \eeqn
We refer to $q_{_J}$ as {\it the polynomial associated with $J$}.

To prove Proposition~\ref{uniq-decom}, we begin with a lemma.

\begin{lemma}
    \label{rep-tuple-uni-full}
Let $d \in \ntwo$, and let $\mf n = (n_j)_{j=1}^d$ be a  $d$-tuple of distinct elements in $\ntwo$ such that the set $\{\ln n_j: j=1, \dots, d\}$ is linearly dependent over $\mathbb Q$, but every proper subset is linearly independent over $\mathbb Q.$ Then, there exists a unique $d$-tuple $(\beta_1,\beta_2, \ldots ,\beta_{d}) \in \mathbb N^{d}$ satisfying $\gcd(\beta_j)_{j=1}^d=1$, and a unique partition $\{I_1, I_2\}$ of $\{1, 2, \ldots, d\}$ such that 
		\beqn
		\prod_{i \in I_{1}}	n_i^{\beta_i} = \prod_{i \in I_{2}} n_i^{\beta_i}.
		\eeqn
\end{lemma}
\begin{proof}
Let $\mathcal S : =\{\ln n_j:j=1,\dots,d\}$. 
 Since $\mathcal S$ is linearly dependent over $\mathbb Q$ and every proper subset is linearly independent, there exist a partition $\{I_1,I_2\}$ of $\{1,\ldots, d\}$ and a tuple $(\alpha_i)_{i=1}^d$ of positive integers such that
\begin{equation}\label{CHid1}
		\prod_{i \in I_1} n_{i}^{\alpha_{i}} = \prod_{i \in I_2} n_{i}^{\alpha_{i}}.
\end{equation} 
 Let $ k:=\gcd(\alpha_i)_{i=1}^d$, and define $\beta_i=\alpha_i/k$ for all $i$.  Then $\gcd(\beta_i)_{i=1}^d = 1$, and we obtain from ~\eqref{CHid1} that 
 
\beq
	\label{rep1}
    \prod_{i \in I_1} n_{i}^{\beta_{i}} = \prod_{i \in I_2} n_{i}^{\beta_{i}}.
\eeq
To prove the uniqueness,  
suppose there exist another tuple $(\gamma_j)_{j=1}^d \in \mathbb N^d$ satisfying $\gcd(\gamma_j)_{j=1}^d = 1$ and another partition $\{\tilde{I_1}, \tilde{I_2}\}$ of $\{1,\dots,d\}$ such that  
	\beq
	\label{rep2}
	\prod_{i \in \tilde{I}_1}	n_i^{\gamma_i} = \prod_{i \in \tilde{I}_2}	n_i^{\gamma_i}.
	\eeq
Without loss of generality, we assume that $I_1 \cap \tilde{I}_1 \neq \emptyset$; the case $I_1\cap \tilde{I_2}\neq \emptyset$ can be treated analogously. We show below that   
	\beqn
	\text{ $I_1 = \tilde{I}_1,~ I_2 = \tilde{I}_2$ and $\gamma_i = \beta_i$ for all integers $1 \Le i \Le d.$}
	\eeqn
Let $m_0 \in I_1 \cap \tilde{I}_1.$ From \eqref{rep1} and \eqref{rep2}, it follows that
	\beq
    \label{two-reps-equal}
	n_{m_0} = \prod_{i \in I_1 \setminus \{m_0\}}	n_i^{-{\beta^{-1}_{m_0}}\beta_i}
	\prod_{i \in I_2} n_i^{{\beta^{-1}_{m_0}}\beta_i} 
	= 
	\prod_{i \in \tilde{I}_1 \setminus  \{m_0\}}	n_i^{-{\gamma^{-1}_{m_0}}\gamma_i}
	\prod_{i \in \tilde{I}_2}	n_i^{{\gamma^{-1}_{m_0}}\gamma_i}.
	\eeq
Suppose there exists $t \in (I_1 \setminus \{m_0\}) \cap \tilde{I}_2$. Since $\mathcal S \setminus\{\ln n_{m_0}\}$ is rationally independent, it follows from \eqref{two-reps-equal} and Lemma~\ref{unique-term} that 
\[-\frac{\beta_t}{\beta_{m_0}} = \frac{\gamma_t}{\gamma_{m_0}}.\] This is a contradiction as all $\beta_i$ and $\gamma_i$ are positive integers. Therefore,
\[
(I_1 \setminus \{m_0\}) \cap \tilde{I}_2\neq \emptyset,\ \text{ and hence } I_1 \setminus \{m_0\} \subseteq \tilde{I_1} \setminus\{m_0\}.
\]
By a symmetric argument, we also have 
\[
I_2 \cap (\tilde{I}_1 \setminus \{m_0\}) = \emptyset,\ \text{ and hence } \tilde{I_1} \setminus\{m_0\}\subseteq I_1 \setminus \{m_0\}. 
\]
Thus we conclude that $I_1 = \tilde{I_1}$, $I_2 = \tilde{I_2}$ and  
	\beq
	\label{CHid2}
		\prod_{i\in I_1}n_{i}^{\gamma_{i}} = \prod_{i\in I_2}n_{i}^{\gamma_{i}}.
	\eeq
It remains to verify that $\beta_i = \gamma_i$ for all $i$. Fix $j\in I_1$. Substituting the expression for $n_{j}$ from \eqref{rep1} into \eqref{CHid2} gives   
    \beqn
\prod_{i\in I_1\setminus\{j\}}n_{i}^{\gamma_{i}-{\beta^{-1}_{j}}\beta_{i}\gamma_{j}}\cdot\prod_{i\in I_2}n_{i}^{-\gamma_{i}+{\beta^{-1}_{j}}\beta_{i}\gamma_{j}}=1.
    \eeqn
Since $\mathcal S \setminus \{\ln n_{j}\}$ is linearly independent over $\mathbb Q$,
	\beq \label{ratio-equal}
	\gamma_{i}=\frac{\gamma_{j}\beta_{i}}{\beta_{j}}, \ i\in I_1\setminus\{j\}\ \text{ and }\ \gamma_{i}=\frac{\gamma_{j}\beta_{i}}{\beta_{j}},\ i\in I_2.
	\eeq
Let $r= \gcd(\gamma_{j}, \beta_{j})$, and define $\tilde{\gamma}_{j} = \frac{\gamma_{j}}{r}$ and $\tilde{\beta}_{j} = \frac{\beta_{j}}{r}$.  
Then ~\eqref{ratio-equal} can be rewritten as
	\beq \label{ratio-equal-modified}
	\gamma_{i} = \frac{\tilde{\gamma}_{j} \beta_{i}}{\tilde{\beta}_{j}}, \  i\in I_1\setminus\{j\}\ \text{ and }\  \gamma_{i} = \frac{\tilde{\gamma}_{j} \beta_{i}}{\tilde{\beta}_{j}},\ i\in I_2.
	\eeq  
 Since $\gcd(\tilde{\gamma}_{j}, \tilde{\beta}_{j}) =1$, it follows from \eqref{ratio-equal-modified} that $\tilde{\beta}_{j}$ divides $\beta_{i}$ for all $i$. Hence $\tilde{\beta}_{j} = 1$, since $\gcd(\beta_{i})_{i=1}^d = 1.$ Therefore, $\beta_{j}= r$ and  $\gamma_{j}=m\beta_{j}$ for some $m \in \mathbb N.$ It then follows from \eqref{ratio-equal} that $m$ divides $\gamma_i$ for all $i$, which forces $m=1$, since $\gcd(\gamma_i)_{i=1}^d=1.$ Therefore, $\gamma_{j}=\beta_{j}$, and substituting back into \eqref{ratio-equal} yields $\gamma_i=\beta_i$ for all $i$. This completes the proof.
\end{proof}

We now prove Proposition~\ref{uniq-decom} as an immediate consequence of Lemma~\ref{rep-tuple-uni-full}.

\begin{proof}[Proof of Proposition~\ref{uniq-decom}]
 Fix $J \in \mathcal J(\mf n)$. By definition, the set $\{\ln n_j : j \in J\}$ is linearly dependent over $\mathbb Q$ and every proper subset is linearly independent. Since $J$ is a finite set, 
Lemma~\ref{rep-tuple-uni-full} applies, and we obtain a unique tuple $(\beta_j)_{j \in J} \in \mathbb N^{|J|}$ with $\gcd(\beta_j)_{j \in J}=1$, and a unique partition $\{J_1, J_2\}$ of $J$ such that 
\beqn
		\prod_{i \in J_1}	n_i^{\beta_i} = \prod_{i \in J_2} n_i^{\beta_i}.
\eeqn     
This completes the proof.
\end{proof}

We are now ready to describe the main result of this section, which determines the multiplier variety associated with CNP Dirichlet series kernels arising from rationally log-dependent frequency data.

 \begin{theorem}\label{variety-gen}
Let $d \in \ntwo \cup \{\infty\}$, and let $(\mf b = (b_j)_{j=1}^d, \mf n= (n_j)_{j=1}^d)$ be a pair of weight and frequency data of length $d$ 
such that $\{\ln n_j: j=1,\dots,d\}$ is linearly dependent over $\mathbb Q$. 
Then, the multiplier variety $V$ associated with $K_{\mf b, \mf n}$ is given by
\beqn
V = \mathbb B_d \cap \bigcap_{J \in \mathcal J(\mf n)} Z(q_{_J}),
 \eeqn
where for each $J\in \mathcal J(\mf n)$ the polynomial
 \begin{equation*}
 q_{_J} (\mf z) = \prod\limits_{i \in J_1}
 b_{i}^{\frac{\beta_i}{2}}\prod\limits_{i \in J_2}z_i^{\beta_i}-\prod\limits_{i \in J_2}b_{i}^{\frac{\beta_i}{2}}\prod\limits_{i \in J_1}z_i^{\beta_i},\quad \mf z = (z_i)_{i=1}^d \in \mathbb C^d,
\end{equation*}
is determined by the unique partition $\{J_1,J_2\}$ of $J$ and the unique $|J|$-tuple $(\beta_j)_{j \in J} \in \mathbb N^{|J|}$ with $\gcd(\beta_j)_{j \in J}=1$, as described in Proposition~\ref{uniq-decom}, and $Z(q_{_J})$ denotes the zero set of $q_{_J}$.
 \end{theorem}

\begin{remark} 

Consider $\mf b = (b_1, b_2, b_3) \in (0, \infty)^3$ and $\mf n = (n_1, n_2, n_3) \in \ntwo^3$ such that $\{\ln n_j: j=1,2, 3\}$ is linearly dependent over $\mathbb Q$. Define $\mf b' = (b_2, b_1, b_3)$ and $\mf n' = (n_2, n_1, n_3)$. Then $K_{\mf b, \mf n}= K_{\mf b', \mf n'}$. However, their associated multiplier varieties may differ. An easy application of Theorem~\ref{variety-gen} shows that either $V=V'$ or $\phi(V)= V'$, where $\phi: \mathbb B_3 \rar\ \mathbb B_3$ is the coordinate permutation automorphism defined by $\phi(z_1, z_2, z_3) = (z_2, z_1, z_3)$. 

More generally, $K_{\mf b, \mf n}= K_{\sigma(\mf b), \sigma(\mf n)}$ for any permutation $\sigma$ of $d$-variables.
But, the multiplier variety depends on the ordering of the tuples $\mf b$ and $\mf n$, and a change in the order by a permutation may lead to a different variety. However, such varieties are biholomorphic via a coordinate permutation automorphism of $\mathbb B_d$. This justifies our terminology of referring to the multiplier variety associated with a CNP Dirichlet series kernel $K_{\mf b, \mf n}$, since it is well defined up to such biholomorphic equivalence.  
\end{remark}

The proof of Theorem~\ref{variety-gen} based on a couple of preparatory lemmas.  
\begin{lemma}
\label{existence-J}
Fix $d \in \ntwo \cup \{\infty\}$. Let $\mf n = (n_i)_{i=1}^d$ be a $d$-tuple of distinct integer in $\ntwo$ such that $\{\ln n_i: i=1,\dots,d\}$ is a linearly dependent set over $\mathbb Q$. Let $\mathcal B \subseteq \{1,\dots, d\}$ be such that $\{\ln n_{i} : i \in \mathcal B\}$ forms a basis of $\textup{span}_\mathbb Q\{\ln n_i: i=1,\dots,d\}$.
Then for every integer $k \in \{1, 2, \ldots, d\} \setminus \mathcal B$,
there exist a unique finite subset $I_k$ of $\mathcal B$ 
and a unique $| I_k|$-tuple $(r_i(k))_{i \in I_k}$ of non-zero rational numbers such that $I_k\cup\{k\}\in \mathcal J(\mf n)$ and 
        \beqn
   n_k= \prod_{i\in I_k}n_i^{r_i (k)}.
        \eeqn
 Furthermore, if $q_{_J}$ is the polynomial associated with $J = I_k \cup \{k\}$, then for any point $(w_i)_{i=1}^d\in Z(q_{_J})$ with $w_{i} \neq 0$ for all $i\in I_k$, we have
\beq
\label{wi-k}
w_{k}&=& \sqrt{b_{k}} \prod_{i\in I_k}b_{i}^{-\frac{r_{i}(k)}{2}}w_{i}^{r_{i}(k)},\quad  k \in \{1, 2, \ldots, d\} \setminus \mathcal B.
\eeq
\end{lemma} 
\begin{proof}
Let $\mathcal T: = \{\ln n_{i}: i\in\mathcal B\}$. Fix $k\in \{1, 2, \ldots, d\} \setminus \mathcal B$. Since $\mathcal T$ is a basis, there exist unique disjoint finite subsets $\mathcal B_1, \mathcal B_2\subseteq \mathcal B$ and a unique tuple of positive integers $(\beta_i)_{i\in \mathcal B_1\cup \mathcal B_2\cup \{k\}}$ with $\gcd(\beta_i)_{i\in \mathcal B_1\cup \mathcal B_2\cup \{k\}}=1$, such that 
\beq\label{ni-rep-1}
n_k^{\beta_k}\prod_{i\in \mathcal B_1}n_i^{\beta_i}= \prod_{i\in \mathcal B_2}n_i^{\beta_i}.
\eeq
Let $I_k := \mathcal{B}_1 \cup \mathcal{B}_2$. Then, the above expression yields 
\beqn
n_k = \prod_{i \in I_k} n_i^{r_i(k)}, \quad \text{where } r_i(k) := \begin{cases}
-\frac{\beta_i}{\beta_k}, & i \in \mathcal{B}_1, \\
\frac{\beta_i}{\beta_k}, & i \in \mathcal{B}_2.
\end{cases}
\eeqn
Observe that $\{\ln n_i: i\in I_k\cup \{k\}\}$ is linearly dependent over $\mathbb Q$, but any proper subset is linearly independent, so $J=I_k\cup \{k\}\in\mathcal J(\mf n)$. Moreover, it follows from \eqref{ni-rep-1} that the unique $|J|$-tuple associated with $J$ is $(\beta_j)_{j\in J}$, and the unique partition associated with $J$ is $\{J_1=\mathcal B_1\cup \{k\}, J_2=\mathcal B_2\}$. Therefore, the polynomial associated to $J$ is given by 
\beqn
q_{_J}(\mf z) =
\prod_{i \in \mathcal B_1\cup\{k\}}b_i^{\frac{\beta_i}{2}}\prod\limits_{i \in \mathcal B_2}z_i^{\beta_i}-\prod\limits_{i \in \mathcal B_2}b_{i}^{\frac{\beta_j}{2}}\prod\limits_{i \in \mathcal B_1\cup\{k\}}z_{i}^{\beta_i}, \quad \mf z \in \mathbb C^d.
\eeqn
Now suppose $\mf w = (w_i)_{i=1}^d\in Z(q_{_J})$ with $w_i \neq 0$ for all $i\in I_k$. Then by the expression of $q_{_J}$,
\beqn
w_k^{\beta_k} = \prod\limits_{i \in \mathcal B_2}b_{i}^{-\frac{\beta_i}{2}}
\prod\limits_{i \in \mathcal B_1\cup\{k\}}
 		b_{i}^{\frac{\beta_i}{2}}\prod\limits_{i \in \mathcal B_1}w_{i}^{-\beta_i}\prod\limits_{i \in \mathcal B_2}w_{i}^{\beta_i}.
\eeqn
Substituting the expression for $r_i(k)$, this becomes 
\[
w_{k}=\sqrt{b_{k}} \prod_{i\in I_k}b_{i}^{-\frac{r_{i}(k)}{2}}w_{i}^{r_{i}(k)},
\]
which proves the desire identity.
\end{proof}

\begin{lemma}\label{exis-J}
Let $p, q \in \mathbb N$. Let $\mf n = (n_i)_{i=1}^{p+q}$ be a $(p+q)$-tuple of distinct elements in $\ntwo$ such that $\{\ln n_i: i=1,\dots, p\}$ is linearly independent over $\mathbb Q$. Suppose that there exist $\{\gamma_i\}_{i=1}^p \subseteq \mathbb Q \setminus \{0\}$ and $\{\gamma_{p+i}\}_{i=1}^q \subseteq \mathbb N$ such that
	\beq
	\label{dep-assumption}
	\prod_{i=1}^p n_i^{\gamma_i} = \prod_{i=p+1}^{p+q} n_{i}^{\gamma_{i}}.
	\eeq
  Then, there exists $J \in \mathcal J(\mf n)$ with the property that 
either $J_1$ or $J_2$ is a subset of $ \{1,2, \ldots, p\}$, where $\{J_1, J_2\}$ is the unique partition associated with $J$, as given in Proposition~\ref{uniq-decom}.
\end{lemma}
\begin{proof}
We first consider the base case, where $q=1$ and $p$ is any natural number.  Then by \eqref{dep-assumption}, the set $\{\ln n_i: i=1,\dots,p+1\}$ is linearly dependent over $\mathbb Q$, while $\{\ln n_i: i=1,\dots, p\}$ is linearly independent. Therefore, $\{\ln n_i: i=1,\dots, p\}$ forms a basis of $\textup{span}_\mathbb Q\{\ln n_j: j=1,\dots, p+1\}$. By Lemma~\ref{existence-J}, there exists
$J \in \mathcal J(\mf n)$ such that $J_1 \cup J_2 = J \subseteq \{1, 2, \ldots, p+1\}$. Consequently, either $J_1\subseteq  \{1,2, \ldots, p\}$ or $J_2\subseteq  \{1,2, \ldots, p\}$, establishing the result for $q=1$.

We now proceed by strong induction on $q$.
Assume the lemma holds for all $p\in\mathbb N$ and for all positive integers less than $k$; that is, for every $p \in \mathbb{N}$ and $1 \Le q < k$, if  a $(p+q)$-tuple $\mf n$ satisfies the hypothesis of the lemma, then there exists $ J \in \mathcal{J}(\mf{n})$ such that one of $ J_1$, $ J_2$ is contained in $ \{1, \dots, p\}$. This is our induction hypothesis.

Let $\mf n =(n_i)_{i=1}^{p+k}$ be a $(p+k)$-tuple of distinct elements in $\ntwo$ which satisfies the hypothesis of the lemma with 
	\beq
	\label{Alzerlem1}
	\prod_{i=1}^p n_i^{\gamma_i} = \prod_{i=p+1}^{p+k} n_{i}^{\gamma_{i}}.
	\eeq
Suppose there exists $J \in \mathcal J(\mf n)$ such that both $J_1$ and $J_2$ are not contained in $\{1, 2, \ldots, p\}.$ Our goal is to construct another element of $\mathcal J(\mf n)$, with the help of $J$, for which one of the partitions is contained in $\{1,\dots,p\}$. By Proposition~\ref{uniq-decom}, suppose the representation associated with $J$ takes the form: 
    \beq
    \label{form-3}
    \prod_{i \in J_1} n_i^{\beta_i} = \prod_{i \in J_2} n_i^{\beta_i}.
    \eeq
Let \beq\label{max-assump}
     \max\Big\{\frac{\beta_{i}}{\gamma_{i}}: i \in J_1, i\Ge p+1\Big\}=\frac{\beta_{j}}{\gamma_{j}}.
     \eeq
By \eqref{Alzerlem1},
	\beq
    \label{mj1-expression}
	n_{j}^{\gamma_{j}} = \prod_{i=1}^p n_{i}^{\gamma_i} \times \prod_{\substack{i=p+1 \\ i\neq j}}^{p+k} n_i^{-\gamma_i}.
	\eeq
Substituting \eqref{mj1-expression} into \eqref{form-3} yields:
	\beqn
\prod_{i\in J_1\setminus\{j\}}n_i^{\beta_i\gamma_j}\prod_{i=1}^p n_{i}^{\gamma_i\beta_j}  \prod_{\substack{i=p+1 \\ i\neq j}}^{p+k} n_i^{-\gamma_i\beta_j}=\prod_{i\in J_2}n_i^{\beta_i\gamma_j}.
	\eeqn
After certain rearrangements, this reduces to 
	\beq \label{rearrangement}
\prod_{\substack{i\in J_1\\ i\Le p}} n_i^{\beta_i\gamma_j}
\prod_{i=1}^p n_{i}^{\gamma_i\beta_j}\prod_{\substack{i\in J_2\\ i\Le p}} n_i^{-\beta_i\gamma_j}  
=
\prod_{\substack{i\in J_2 \\ i \Ge p+1}} n_i^{\beta_i\gamma_j}\prod_{\substack{i\in J_1\setminus\{j\} \\ i \Ge p+1}} n_i^{\gamma_i\beta_j-\beta_i\gamma_j}\prod_{\substack{i=p+1 \\ i\notin J_1}}^{p+k} n_i^{\gamma_i\beta_j}.
    \eeq
By \eqref{max-assump}, for all $i \in (J_1\setminus\{j\})\cap \{p+1,\dots, p+k\}$, we have $\gamma_i\beta_{j}-\beta_{i}\gamma_j \Ge 0$, so all exponents on the right-hand side of the above identity are non-negative. Since $J_2$ is not contained in $\{1,2,\dots, p\}$, the set $J_2\cap \{p+1, p+2,\dots, p+k\}$ is non-empty. Also, for each $i \in J_2\cap \{p+1, p+2,\dots, p+k\}$ the exponent of $n_i$ on the right-hand side of ~\eqref{rearrangement} is positive. Consequently, not all the exponents of $n_i$ on the left-hand side of ~\eqref{rearrangement} can be zero. We now define
\begin{align*}
I_1:=\{i: 1\Le i\Le p, \mbox{ the exponent of } n_i\ \mbox{on the left hand-side of ~\eqref{rearrangement} is  non-zero}\}
\end{align*}
and 
\[
I_2:=\{i: p+1\Le i\Le p+k,\ i\neq j,  \ \mbox{ the exponent of } n_i\ \mbox{on the right hand-side of ~\eqref{rearrangement} is positive}\}.  
\]
Consider $ \mf{n}' := (n_i)_{i \in I_1\cup I_2}$. The identity ~\eqref{rearrangement} implies a relation of the form \eqref{dep-assumption} for $ \mf{n}'$. Hence, by the induction hypothesis, there exists $ J' \in \mathcal{J}(\mf{n}')$ such that one of the partitions $J_1', J_2'$ lies entirely in $ I_1$.
Since $\mf n'$ is a sub-tuple of $\mf n$, $J'$ will induce an element in $\mathcal J(\mf n)$ with the required property. This completes the proof.
\end{proof}

We are now ready to prove Theorem~\ref{variety-gen}.

\begin{proof}[\textbf{Proof of Theorem~\ref{variety-gen}}]
Since $\{\ln n_i: i=1,\dots,d\}$ is linearly dependent over $\mathbb Q$, it follows that $\mathcal J(\mf n) \neq \emptyset$. To prove the inclusion 
\[V \subseteq \mathbb B_d \cap \Big(\bigcap_{J \in \mathcal J(\mf n)} Z(q_{_J})\Big),\] 
let $\mf w \in V$, and fix $J\in \mathcal J(\mf n)$. By Proposition~\ref{uniq-decom}, let $(\beta_j)_{j \in J}$ with $\gcd(\beta_j)_{j=1}^d=1$ and $\{J_1, J_2\}$ be the unique $|J|$-tuple and the unique partition associated with $J$, respectively. Then, the representation associated to $J$ is  
\beq
\label{uniq-rep-mainproof}
\prod_{i \in J_1}	n_i^{\beta_i} = \prod_{i \in J_2} n_i^{\beta_i},
\eeq
and the polynomial associated to $J$ is given by 
\beqn
	q_{_J}(\mf z) = \prod\limits_{i \in J_1}
	b_i^{\frac{\beta_i}{2}}\prod\limits_{i \in J_2}z_i^{\beta_i}-\prod\limits_{i \in J_2}b_i^{\frac{\beta_i}{2}}\prod\limits_{i \in J_1}z_i^{\beta_i},\quad \mf z \in \mathbb C^d.
	\eeqn
 Let the kernel $K_{\mf b, \mf n}$ be defined on $\mathbb H_\rho$ for some $\rho\in \mathbb R$. Recall that 
    \beq
    \label{f-s}
    f_{\mf b, \mf n}(s) = (\sqrt{b_i}n_i^{-s})_{i=1}^d,\quad  s \in \mathbb H_\rho.
    \eeq
Since 
\beqn
	q_{_J}(f_{\mf b, \mf n}(s)) = \prod\limits_{i \in J_1}
	b_i^{\frac{\beta_i}{2}}\prod\limits_{i \in J_2}b_i^{\frac{\beta_i}{2}}n_i^{-\beta_is}-\prod\limits_{i \in J_2}b_i^{\frac{\beta_i}{2}}\prod\limits_{i \in J_1}b_i^{\frac{\beta_i}{2}}n_i^{-\beta_is}\overset{\eqref{uniq-rep-mainproof}}= 0,\quad s \in \mathbb H_\rho,
    \eeqn
which implies that $q_{_J}$ vanishes on $f_{\mf b, \mf n}(\mathbb H_\rho).$  Since $q_{_J}$ is a polynomial and thus an element of $\mathcal M_d$, it follows from the definition of $V$ (see \eqref{def-V}) that $q_{_J}(\mf w) = 0$ and $\mf w \in \mathbb B_d$. As $J \in \mathcal J(\mf n)$ was arbitrary, we conclude that 
$\mf w \in \mathbb B_d \cap \bigcap_{J \in \mathcal J(\mf n)} Z(q_{_J}),$ 
establishing the desired conclusion.
	
We now prove the reverse inclusion: 
\[
    \mathbb B_d \cap \bigcap_{J \in \mathcal J(\mf n)} Z(q_{_J}) \subseteq V.
\]
To this end, let 
\beq
    \label{assump-standing}
    \mf w = (w_i)_{i=1}^d\in \mathbb B_d \cap \bigcap_{J \in \mathcal J(\mf n)} Z(q_{_J}),\ \text{ and }\ g \in \mathcal M_d \text{ such that } g|_{f_{\mf b, \mf n}(\mathbb H_\rho)} \equiv 0.
\eeq
We aim to show that $g(\mf w) = 0$, which will imply that $\mf w\in V$. By Remark~\ref{0-in-V}, if $\mf w = 0$ then $\mf w \in V.$ So, we assume that $\mf w \neq 0$. 
    
Set $I: = \{1 \Le i \Le d : w_i \neq 0\} \neq \emptyset$. Let $\mathcal B\subseteq I$ (possibly infinite) be such that $\{\ln n_i : i \in \mathcal B\}$ forms a basis of $\textup{span}_\mathbb Q\{\ln n_i : i \in I\}.$ 
Then for every $k \in I \setminus \mathcal B$ (if, non-empty), the set $\{\ln n_k, \ln n_i : i \in \mathcal B\}$ is linearly dependent over $\mathbb Q$. So, by Lemma~\ref{existence-J}, there exists a tuple $(r_i(k))_{i\in \mathcal B}$ of rational numbers such that
\beq
\label{n-it-formula}
n_{k}=\prod_{i\in\mathcal{B}} n_{i}^{r_i(k)}, \quad k \in I \setminus \mathcal B.
\eeq
Similarly, let $I_0: = \{1 \Le i \Le d : w_i = 0\}$. If $I_{0}$ is non-empty, then there exists a non-empty set $\mathcal B_0 \subseteq I_0$ (possibly infinite) such that $\{\ln n_i : i \in \mathcal B_0 \cup \mathcal B\}$ forms a basis of $\textup{span}_\mathbb Q\{\ln n_i : i =1, 2, \ldots, d\}$. Note that if $\clb_0$ is empty, then it contradicts the identity \eqref{wi-k} in Lemma~\ref{existence-J} due to the non-emptiness of $I_0$.
For each $k \in I_0 \setminus \mathcal B_0$ (if, non-empty), the set $\{\ln n_k, \ln n_i : i \in \mathcal B \cup \mathcal B_0\}$ is linearly dependent over $\mathbb Q$, and hence, by Lemma~\ref{existence-J}, there exist tuples $(r_i(k))_{i\in\mathcal{B}}$ and $(\tilde{r}_i(k))_{i\in\mathcal{B}_0}$ of rational numbers such that
\beq
\label{n-it-formula-1}
n_{k}=\prod_{i\in\clb} n_{i}^{r_i(k)}\prod_{i\in\clb_0}n_i^{\tilde{r}_i(k)}, \quad k \in I_0 \setminus \mathcal B_0.
\eeq

Thus, for any $\alpha = (\alpha_i)_{i=1}^d\in \mathbb Z_{+}^d,$ putting \eqref{n-it-formula} and \eqref{n-it-formula-1} together, we get
\beq
\label{prod-factor}
\prod_{i=1}^d n_i^{\alpha_i} &=& \prod_{i\in\clb} n_{i}^{\alpha_{i}} \times \prod_{i\in\clb_0} n_{i}^{\alpha_{i}} \times  \prod_{k \in I \setminus \mathcal B} n_k^{\alpha_k} \times \prod_{k \in I_0 \setminus \mathcal B_0} n_k^{\alpha_k} \notag\\
&=& \prod_{i\in\clb} n_{i}^{\big(\alpha_{i}+\sum\limits_{k \in I\setminus \mathcal B}\alpha_{k}r_i(k)+\sum\limits_{k \in I_0 \setminus \mathcal B_0} \alpha_{k}r_i(k)\big)} 
\times \prod_{i\in\clb_0} n_{i}^{\big(\alpha_{i}+\sum\limits_{k \in I_0 \setminus \mathcal B_0} \alpha_{k}\tilde{r}_i(k)\big)} \notag \\
&=&   \prod_{i\in\clb}n_{i}^{\psi_1(\alpha)(i)} \times \prod_{i\in\clb_0}n_{i}^{\psi_2(\alpha)(i)},
\eeq
where $\psi_1: \mathbb Z_+^d \rar \mathbb Q^{\clb}$ (the set of all functions from $\clb$ to $\mathbb Q$) and $\psi_2: \mathbb Z_+^d \rar \mathbb Q^{\clb_0}$ are defined by
\beqn
\psi_1(\alpha)(i) = \alpha_{i}+\sum\limits_{k \in I\setminus \mathcal B}\alpha_{k}r_i(k)+\sum\limits_{k \in I_0 \setminus \mathcal B_0} \alpha_{k}r_i(k), \quad   i\in\clb,~~\alpha \in \mathbb Z_+^d,
\eeqn
and 
\beqn
\psi_2(\alpha)(i) = \alpha_{i}+\sum\limits_{k \in I_0 \setminus \mathcal B_0} \alpha_{k}\tilde{r}_i(k), \quad i\in\clb_0,~~ \alpha \in \mathbb Z_+^d.
\eeqn
Observe that both $I \setminus \mathcal B$ and $I_0 \setminus \mathcal B_0$ can not be empty. If either of the sets $I \setminus \mathcal B$ and $I_0 \setminus \mathcal B_0$ is empty, we adopt the convention that the corresponding summation terms are omitted from the definitions of $\psi_1$ and $\psi_2$, respectively.   
Consider the function $\psi:\mathbb Z_+^d \rar \mathbb Q^{\clb\cup\clb_0}$ defined by 
\[ \psi(\alpha)(i) = \begin{cases} \psi_1(\alpha)(i) & \text{if}\ i\in\clb,\\
\psi_2(\alpha)(i) & \text{if}\ i\in\clb_0.
    \end{cases} \] 

We now pause to state and prove the following claim, crucial for the final step of the proof.

\textbf{Claim:}
\beq
\label{imp-obs-psi}
\psi(\{\alpha \in \mathbb Z_+^d : \alpha_i = 0~\text{for all}~i \in I_0\}) = \{\mf y\in \text{Im}(\psi): \mf y(i)=0\ \text{for all}\ i\in\clb_0 \}.
\eeq
\textit{Proof of the claim.} The inclusion ``$\subseteq$" is immediate by the definition of $\psi$. For the reverse inclusion, fix $\mf y \in \text{Im}(\psi)$ with $\mf y(i)=0$ for all $i\in \clb_0$ and let $\alpha \in \psi^{-1}(\mf y).$ 
 Then $\psi_2(\alpha) = 0,$ and by \eqref{prod-factor}, 
 \beqn
\prod_{i=1}^d n_{i}^{\alpha_{i}} = \prod_{i\in\clb} n_{i}^{\mf y(i)}.
 \eeqn
 Using this and the fact that $I \cup I_0 = \{1, 2, \ldots, d\}$, we get
\beq
\label{exis-J-cons}
\prod_{i \in I_0} n_{i}^{\alpha_{i}} = \prod_{i \in I \setminus \mathcal B} n_i^{-\alpha_i} \times \prod_{i\in\clb} n_{i}^{\mf y(i) - \alpha_{i}} \overset{\eqref{n-it-formula}}= \prod_{i\in\clb}n_{i}^{{\mf y(i)}-\alpha_{i}-\sum\limits_{k \in I \setminus \mathcal B}\alpha_{k}r_i{(k)}}.
\eeq
We need to verify that $\alpha_i = 0$ for all $i \in I_0$. Suppose in contrast that $\alpha_{j} \neq 0$ for some $j \in I_0$. Since $\{\ln n_i: i \in \mathcal B\}$ is linearly independent over $\mathbb Q$, Lemma~\ref{exis-J} is applicable to the relation in \eqref{exis-J-cons}. Thus, there exists $J \in \mathcal J(\mf n)$ such that $J \subseteq \mathcal B \cup I_0$ and either $J_1$ or $J_2$ is a subset of $\mathcal B$, where $\{J_1, J_2\}$ is the partition associated with $J$. Without loss of generality, we assume that $J_1\subseteq \mathcal B$. In such a case, if $J_2 \subseteq \mathcal B$, then $J = J_1 \cup J_2 \subseteq \mathcal B$ contradicts the fact that $\{\ln n_{i}\}_{i \in \mathcal B}$ is linearly independent over $\mathbb Q$. Hence, there exists $j_0 \in I_0 \cap J_2$. 
Consequently, the polynomial associated with $J$ has the form 
\beqn
	q_{_J}(\mf z) =z_{j_0}^{\beta_{j_0}}\prod\limits_{i \in J_1}
	b_i^{\frac{\beta_i}{2}}\prod\limits_{i \in J_2 \setminus \{j_0\}}z_i^{\beta_i}-\prod\limits_{i \in J_2}b_i^{\frac{\beta_i}{2}}\prod\limits_{i \in J_1}z_i^{\beta_i}, \quad \mf z \in \mathbb C^d,
\eeqn
where $(\beta_i)_{i \in J} \in \mathbb N^{|J|}$ satisfying $\gcd(\beta_i)_{i \in J}=1$ is the tuple associated with $J$.
By assumption ~\eqref{assump-standing}, we have $q_{_J}(\mf w) = 0$. Since $j_0\in I_0$, $w_{{j_0}} = 0$ and it follows from the expression of $q_{_J}$ that $w_{i}$ must be zero for some $i \in J_1 \subseteq  \mathcal B$. This is a contradiction, as by definition $w_{i} \neq 0$ for all $i \in \mathcal B$. This proves the claim. 

Now coming back to the proof of the theorem, note that by ~\eqref{assump-standing} $g \in H_d^2$ and admits a power series expansion $
g(\mf z)=\sum\limits_{{\bf \alpha} \in \mathbb Z_+^d}c_{\mf \alpha} {\mf z}^{\mf \alpha},~ \mf z \in \mathbb B_d$. 
Since $g|_{f_{\mf b,\mf n}(\mathbb H_\rho)} \equiv 0$, then using  ~\eqref{prod-factor}, ~\eqref{g-f} and the definition of $\psi$, we have for all $s \in \mathbb H_\rho$,
\beq
\label{g-f}
g(f_{\mf b, \mf n}(s))& \overset{\eqref{f-s}}= & \sum_{\alpha \in \mathbb Z_+^d}c_\alpha \mf b^{\frac{\alpha}{2}} \Big(\prod_{i=1}^d n_i^{\alpha_i}\Big)^{-s}\notag \\
&=& \sum_{{\mf \alpha} \in \mathbb Z_+^d}c_{\mf \alpha}\mf b^{\frac{\mf \alpha}{2}}\Big(\prod_{i\in\clb}n_{i}^{\psi_1(\alpha)(i)} \times \prod_{i\in\clb_0}n_{i}^{\psi_2(\alpha)(i)}\Big)^{-s}\notag\\
&=&\sum_{\substack{\mf y \in \text{Im}(\psi)}}\Big(\sum_{\substack{\alpha \in \mathbb Z_+^d
		\\ \psi(\alpha) = \mf y}}
c_{\alpha} \mf b^{\frac{\alpha}{2}}\Big)  \Big(\prod_{i\in\clb}
n_{i}^{\mf y(i)} \times \prod_{i\in\clb_0}
n_{i}^{\mf y(i)}\Big)^{-s}  = 0,
\eeq
where $\mf b^{\frac{\alpha}{2}} = \prod_{i=1}^d b_{i}^{\frac{\alpha_i}{2}}$ for any $\alpha = (\alpha_i)_{i=1}^d \in \mathbb Z_+^d$. To see that the series above is a Dirichlet series, note that if $\mf y \in \text{Im}(\psi)$, then for every $(\alpha_j)_{j=1}^d \in \psi^{-1}(\mf y),$
\beqn
u_\mf y: = \prod_{i\in\clb}n_{i}^{\mf y(i)} \times \prod_{i\in\clb_0}
n_{i}^{\mf y(i)} \overset{\eqref{prod-factor}}= \prod_{i=1}^d n_{i}^{\alpha_{i}} \in \mathbb N.
\eeqn
Moreover, since $\{\ln n_i : i \in \mathcal B_0 \cup \mathcal B\}$ is linearly independent over $\mathbb Q$, according to Lemma~\ref{unique-term}, each $\mf y \in \text{Im}(\psi)$ corresponds to a unique natural number $u_\mf y$. Therefore, the series appearing in ~\eqref{g-f} is a Dirichlet series. So,
by the uniqueness of Dirichlet series,
\beq
 \label{cond-uniq-coeff}
\sum_{\substack{\alpha \in \mathbb Z_+^d
\\ \psi(\alpha) = \mf y}}
c_{\alpha} \mf b^{\frac{\alpha}{2}} = 0 \quad \forall ~~\mf y \in \text{Im}(\psi).
\eeq
In particular, for $\mf y \in \text{Im}(\psi)$ with $\mf y(i)=0$ for all $i\in\clb_0$, by the previous claim in \eqref{imp-obs-psi}, we get
\beq
\label{cond-uniq-coeff-red}
\sum_{\substack{\alpha \in \psi^{-1}(\mf y)
		\\ 	\alpha_{i}=0, \, i \in I_0}}
c_{\mf  \alpha} \mf b^{\frac{\alpha}{2}} = \sum_{\substack{\alpha \in \mathbb Z_+^d
\\ \psi(\alpha) = \mf y}}
c_{\alpha} \mf b^{\frac{\alpha}{2}} = 0.
\eeq
Rest of the proof is divided into the following two cases.

\textbf{Case I.} Suppose $I_0\neq\emptyset$. Since $w_i = 0$ for $i \in I_0,$ 
by Lemma~\ref{existence-J}, 
\beqn
g(\mf w) = \sum\limits_{\substack{{\mf \alpha} \in \mathbb Z_+^d\\ \alpha_{i}=0,\,i \in I_0}}c_{\mf\alpha} \mf w^{\mf \alpha} &=& \sum\limits_{\substack{{\mf \alpha} \in \mathbb Z_+^d\\ \alpha_{i}=0, \,i \in I_0}}c_{\mf  \alpha} \times  \prod\limits_{i\in\clb} w_{i}^{\alpha_{i}}\times\prod\limits_{k \in I \setminus \mathcal B}\Big(b_{k}^{\frac{\alpha_{k}}{2}}\prod\limits_{i\in\clb} b_{i}^{-\frac{r_i(k)\alpha_k}{2}}w_{i}^{{r_i(k)\alpha_k}}\Big) \notag\\
&=& \sum_{\substack{{\mf \alpha} \in \mathbb Z_+^d\\ \alpha_{i}=0, \,i \in I_0}}c_{\mf  \alpha} 
\times \prod\limits_{i \in I} b_{i}^{\frac{\alpha_{i}}{2}} \times \prod\limits_{i\in\clb} (b_{i}^{-\frac{1}{2}} w_{i})^{\alpha_{i}+\sum\limits_{k \in I \setminus \mathcal B}\alpha_{k}r_i(k)} \notag\\
&\overset{\eqref{imp-obs-psi}}=& 
\sum_{\substack{\mf y \in \text{Im}(\psi) \\ \mf y(i) = 0 (\forall\  i\in\clb_0)}} 
\Big(\sum_{\substack{\alpha \in \psi^{-1}(\mf y)
		\\ 	\alpha_{i}=0, \, i \in I_0}}
c_{\mf  \alpha} \prod_{i=1}^d b_i^{\frac{\alpha_i}{2}}\Big)
 \prod_{i\in\clb} (b_{i}^{-\frac{1}{2}} w_{i})^{\mf y(i)}\\
 &\overset{\eqref{cond-uniq-coeff-red}}=& 0.
\eeqn
\textbf{Case II.} Suppose $I_0=\emptyset$. Then, for $\alpha = (\alpha_i)_{i=1}^d\in\mathbb Z_+^d$,
\beqn
\psi(\alpha)(i) = \psi_1(\alpha)(i)=\alpha_{i}+\sum\limits_{k \in I\setminus \mathcal B}\alpha_{k}r_i(k), \quad  i\in\clb,
\eeqn
so that for $\mf y \in  \text{Im}(\psi)$,  $\mf y(i) = \psi(\alpha)(i), ~ i \in \mathcal B.$
Therefore, by Lemma~\ref{existence-J}, 
\beqn
g(\mf w) = \sum_{\substack{{\mf \alpha} \in \mathbb Z_+^d}}c_{\mf\alpha} \mf w^{\mf \alpha} &=& \sum_{\substack{{\mf \alpha} \in \mathbb Z_+^d}}c_{\mf  \alpha} \times \prod_{i\in\clb} w_{i}^{\alpha_{i}}\times\prod_{k \in I \setminus \mathcal B}\Big(b_{k}^{\frac{\alpha_{k}}{2}}\prod_{i\in\clb} b_{i}^{-\frac{r_i(k)\alpha_k}{2}}w_{i}^{{r_i(k)\alpha_k}}\Big) \notag\\
&=& \sum_{\substack{{\mf \alpha} \in \mathbb Z_+^d}}c_{\mf  \alpha} 
\times \prod_{i \in I} b_{i}^{\frac{\alpha_{i}}{2}} \times \prod_{i\in\clb} (b_{i}^{-\frac{1}{2}} w_{i})^{\alpha_{i}+\sum\limits_{k \in I \setminus \mathcal B}\alpha_{k}r_i(k)} \notag\\
&=& 
\sum_{\substack{\mf y \in \text{Im}(\psi)}} 
\Big(\sum_{\substack{\alpha \in \psi^{-1}(\mf y)
		}}
c_{\mf  \alpha} \prod_{i=1}^d b_i^{\frac{\alpha_i}{2}}\Big)
 \prod_{i\in\clb} (b_{i}^{-\frac{1}{2}} w_{i})^{\mf y(i)}\\
 &\overset{\eqref{cond-uniq-coeff}}=& 0.
\eeqn
 This completes the proof.
\end{proof}


We now consider a family of examples of CNP Dirichlet series kernels and compute their multiplier varieties to illustrate Theorem~\ref{variety-gen}.
\begin{example}
(i) Consider $d= 3$ and $\mf n= (2, 3, 6)$. Let $\mf{b} = (b_j)_{j=1}^3$ be a triple of positive real numbers. Then, $\mathcal J(\mf n) = \{\{1,2,3\}\}$, and the unique partition and the triple associated with $\{1,2,3\}$ are $\{\{1,2\}, \{3\}\}$ and $(1,1,1)$, respectively. By Theorem~\ref{variety-gen}, the multiplier variety associated with $K_{\mf b, \mf n}$ is given by
   \beqn
V &=& \{(z_1, z_2,z_3) \in \mathbb B_3 :  \sqrt{b_1b_2}z_3 =\sqrt{b_3}z_1z_2\}.
   \eeqn
   
Now consider extending $\mf n$ by including additional primes; that is, suppose that the frequency data is a tuple of distinct elements in $\ntwo$ of the form $\mf n = (2, 3, 6, p_1, \dots, p_r)$, where $p_i$'s are distinct prime numbers larger than $3$. Then, for any weight tuple $\mf b=(b_i)_{i=1}^{r+3}$, the multiplier variety associated with $K_{\mf b,\mf n}$ is determined by the same polynomial and it is given by  
\beqn
V &=& \{(z_1,z_2,\dots, z_{r+3}) \in \mathbb B_{r+3} :  \sqrt{b_1b_2}z_3 =\sqrt{b_3}z_1z_2\}.
   \eeqn
This is due to the fact that $\mathcal J(\mf n)$ remains the same, as do the unique partition and the tuple associated with each member of $\mathcal J(\mf n)$.  

 (ii) Consider $d= 3$ and $\mf n = (q_1, q_2, q_1^2 q_2)\in \ntwo^3$ such that $\{\ln q_1, \ln q_2\}$ is a linearly independent set over $\mathbb Q$. Let $\mf{b} = (b_j)_{j=1}^3$  be a weight data. Then, $\mathcal J(\mf n) = \{\{1,2,3\}\}$, and the unique partition and the triple associated with $\{1,2,3\}$ are $\{\{1,2\}, \{3\}\}$ and $(2,1,1)$, respectively.
 Then, by Theorem~\ref{variety-gen}, the multiplier variety associated with $K_{\mf b, \mf n}$ is given by
   \beqn
V 
&=& \{(z_j)_{j=1}^3 \in \mathbb B_3 :  b_1\sqrt{b_2}z_3 =\sqrt{b_3}z_1^2z_2\}.
   \eeqn
\end{example}

If $V$ is the multiplier variety of a CNP kernel $K_{\mf b, \mf n}$ defined on $\mathbb H_\rho$, then by Remark~\ref{0-in-V}, the zero vector is always contained in $V$, and from the definition of multiplier variety \eqref{def-V} 
$f_{\mf b, \mf n}(\mathbb H_\rho) \subseteq V$.
Thus, $\{ 0\} \cup f_{\mf b, \mf n}(\mathbb H_\rho) \subseteq V$. As an application of Theorem~\ref{variety-gen}, the following result provides a family of CNP Dirichlet series kernels for which there is not much difference between $V$ and $f_{\mf b, \mf n}(\mathbb H_\rho)\cup\{ 0\}$ (cf. \cite[Lemma~8.1]{DHS}).

\begin{corollary}
Let $d \in \ntwo \cup \{\infty\}$, and let $(\mf b = (b_j)_{j=1}^d, \mf n= (n_j)_{j=1}^d)$ be a pair of weight and frequency data such that 
for all $i=2,\dots,d$, the set $\{\ln n_1, \ln n_i\}$ is linearly dependent over $\mathbb Q$. Let $V$ denote the multiplier variety associated with $K_{\mf b, \mf n}$. Suppose that $\sigma =\sigma_\mf a\big(\sum_{j=1}^d b_j n_j^{-s}\big)$.
\begin{itemize}
    \item [(i)] If $\sum_{j=1}^d b_jn_j^{-\sigma} \Ge 1$, then $V = f_{\mf b, \mf n}(\mathbb H_{\frac{\rho}{2}})\cup \{0\}$, where $\rho$ is the unique real number satisfying $\sum_{j=1}^d b_jn_j^{-\rho} = 1$.
    \item [(ii)] If instead, $\sum_{j=1}^d b_jn_j^{-\sigma} < 1$, then $V = f_{\mf b, \mf n}(\overline{\mathbb H_{\frac{\sigma}{2}}})\cup \{0\}$.
    \end{itemize}
\end{corollary}
\begin{proof}
(i) Suppose that $\sum_{j=1}^d b_jn_j^{-\sigma} \Ge 1$. Then, $K_{\mf b,\mf n}$ is well defined on $\mathbb H_{\frac{\rho}{2}}$. Hence, by the definition of multiplier variety, $f_{\mf b, \mf n}(\mathbb H_\frac{\rho}{2}) \cup \{0\} \subseteq V.$ 
Thus, it remains to show that $V \subseteq f_{\mf b, \mf n}(\mathbb H_\frac{\rho}{2}) \cup \{0\}$.
To this end, let $0 \neq \mf w= (w_j)_{j=1}^d \in V.$ Then, by Theorem~\ref{variety-gen}, we have
\beq \label{w-in-common}
\mf w \in \bigcap_{J \in \mathcal J(\mf n)} Z(q_{_J}) \cap \mathbb B_d.
\eeq
By the assumption, for every integer $i = 2, \ldots, d$, the set $J_i: =\{1, i\} \in \mathcal J(\mf n)$. Hence, by Proposition~\ref{rep-tuple-uni-full}, there exist integers $x_i, y_i \in \mathbb N^2$ with $\gcd(x_i, y_i) = 1$ such that 
\beq\label{cor-var-gen1}
n_i^{x_i}=n_1^{y_i}, \quad i =2, \ldots, d.
\eeq
If $q_{_{J_i}}$ is the polynomial associated to $J_i$, then by \eqref{w-in-common} we have 
\beqn
q_{_{J_i}}(\mf w) = b_i^{\frac{x_i}{2}}w_1^{y_i}-b_1^{\frac{y_i}{2}}w_i^{x_i} =  0, \quad i=2,\dots,d.
\eeqn
Since $\mf w \neq 0$, it follows that $w_1 \neq 0$. Solving for $w_i$, we obtain
\beq\label{cor-var-gen2}
w_i=\sqrt{\frac{b_i}{b_1^{r_i}}}w_1^{r_i}, \quad i = 2, \ldots, d,
\eeq
where $r_i = \frac{y_i}{x_i}> 0$.
Moreover, since $\sum_{j=1}^d|w_j|^2< \sum_{j=1}^d b_jn_j^{-\rho} = 1,$ it follows that $|w_i|<\sqrt{b_i}n_i^{\frac{-\rho}{2}}$ for some $1\Le i\Le d$.  Then, by \eqref{cor-var-gen2} and the positivity of $r_i$, we conclude that $|w_1|<\sqrt{b_1}n_1^{\frac{-\rho}{2}}$. Define 
\[
s:=-\frac{\ln (b_{1}^{-\frac{1}{2}}w_1)}{\ln n_1}.
\]
Then $s \in \mathbb H_\frac{\rho}{2}$ and $w_1 = \sqrt{b_{1}} n_1^{-s}$. Using \eqref{cor-var-gen1} and \eqref{cor-var-gen2}, we also have 
\beqn
w_i = \sqrt{b_i} (n_1^{r_i})^{-s} = \sqrt{b_i} n_i^{-s},\quad i=2, \ldots, d.
\eeqn 
Hence $\mf w=f_{\mf b, \mf n}(s)\in f_{\mf b, \mf n}(\mathbb H_\frac{\rho}{2})$, showing that $V \subseteq f_{\mf b, \mf n}(\mathbb H_\frac{\rho}{2})\cup\{0\}$. 

(ii) Suppose that $\sum_{j=1}^d b_jn_j^{-\sigma} < 1$. Then the kernel $K_{\mf b, \mf n}$ is well defined on $\mathbb H_{\frac{\sigma}{2}}$, and  
\[f_{\mf b, \mf n}(\overline{\mathbb{H}_\frac{\sigma}{2}}) \subseteq \mathbb{B}_d, \text{ and therefore } f_{\mf b, \mf n}(\overline{\mathbb{H}_\frac{\sigma}{2}}) \cup \{{0}\} \subseteq V.\]
We first claim that 
\[\overline{f_{\mf b, \mf n}(\mathbb{H}_\frac{\sigma}{2})} = f_{\mf b, \mf n}(\overline{\mathbb{H}_\frac{\sigma}{2}}) \cup \{0\}.\]
It follows from the continuity of $f_{\mf b, \mf n}$ that $f_{\mf b, \mf n}(\overline{\mathbb{H}_\frac{\sigma}{2}}) \cup \{0\} \subseteq \overline{f_{\mf b, \mf n}(\mathbb{H}_\frac{\sigma}{2})}$. For the reverse inclusion, let $\mathbf{z} = (z_j)_{j=1}^d \in \overline{f_{\mf b, \mf n}(\mathbb{H}_\frac{\sigma}{2})} \setminus \{0\}$. Then there exists a sequence $\{s_m\}_{m=1}^\infty \subseteq \mathbb{H}_\frac{\sigma}{2}$ such that 
\beqn
\lim_{m \rar \infty} \sqrt{b_i} n_i^{-s_m} = z_i,\quad i = 1, \ldots, d.
\eeqn
Since $\mathbf{z} \neq 0$, there exists some $l \in \{1, \ldots, d\}$ such that $z_{l} \neq 0$. It follows that 
\[
\lim_{m \rar \infty} s_m=\frac{-1}{\ln n_{l}} \ln \frac{z_l}{\sqrt{b_l}}:=s\in \overline{\mathbb H_\frac{\sigma}{2}},
\]
and therefore 
\beqn
z_i=\lim_{m \rar \infty} \sqrt{b_i} n_i^{-s_m} = \sqrt{b_i} n_i^{-s}, \quad i = 1, \ldots, d.
\eeqn
Hence $\mf z = (\sqrt{b_i} n_i^{-s})_{i=1}^d = f_{\mf b, \mf n}(s)$, proving the claim.

Now, to prove the reverse inclusion 
\[V \subseteq f_{\mf b, \mf n}(\overline{\mathbb H_\frac{\sigma}{2}}) \cup \{0\},\]
let $0 \neq \mf w= (w_j)_{j=1}^d \in V$. By similar arguments as in part (i), we have 
$w_1 \neq 0$ and 
\beq\label{cor-var-gen2-new}
w_i=\sqrt{\frac{b_i}{b_1^{r_i}}}w_1^{r_i}, \quad i = 2, \ldots, d,
\eeq
where $r_i = \frac{y_i}{x_i}$.
Let 
\[s:=-\frac{\ln b_{1}^{-\frac{1}{2}}w_1}{\ln n_1}.\]
Then using \eqref{cor-var-gen1} and \eqref{cor-var-gen2-new}, we have 
\beq
\label{w-s}
w_i = \sqrt{b_i} (n_1^{r_i})^{-s} = \sqrt{b_i} n_i^{-s},\quad i=2, \ldots, d.
\eeq 
Therefore,
\[\sum_{i=1}^d b_i n_i^{-2\text{Re}(s)}= \sum_{i=1}^d |w_i|^2 < 1,\] 
which, together with the fact that the abscissa of absolute convergence of the Dirichlet series $\sum_{i=1}^db_in_i^{-u}$ is $\sigma$, implies $s \in \overline{\mathbb H_\frac{\sigma}{2}}$. 
Hence, by \eqref{w-s}, $\mf w=f_{\mf b,\mf n}(s) \in f_{\mf b, \mf n}(\overline{\mathbb H_\frac{\sigma}{2}})$, and $V \subseteq f_{\mf b, \mf n}(\overline{\mathbb H_\frac{\sigma}{2}})\cup\{0\}$. 
\end{proof}
We can prove the converse of the above corollary in many particular cases, and we believe that the converse is true in general. However, a general proof remains unknown.

\section{Isometric isomorphism of multiplier algebras}\label{5}


 Given a subset $S\subseteq\mathbb C^d$, the {\em affine hull} of $S$, denoted by $\text{aff}(S)$, is defined as 
 \[\text{aff}(S): = \lambda + span(S-\lambda)\] for any $\lambda \in S$. Note that this definition of an affine hull is independent of the choice of $\lambda$ (see \cite[Section~2.4]{Rudin}). 
The dimension and codimension of the affine hull of $S$ are given by the dimensions of the space $\text{span}(S-\lambda)$ and the quotient space $\mathbb C^d \slash \text{span}(S-\lambda)$, respectively, for any $\lambda \in S$. It turns out that, for isometric isomorphisms of multiplier algebras, the affine dimension and the affine codimension of the associated multiplier variety are a numerical invariant. This motivates us to compute the affine dimensions of multiplier varieties.

\begin{lemma}\label{affine-dim}
Let $d \in \mathbb N \cup \{\infty\}$, and let $(\mf b, \mf n)$ be a pair of weight and frequency data of length $d$.  If $V$ is the multiplier variety associated with $K_{\mf b, \mf n}$, then the dimension and the codimension of $\text{aff}(V)$ are $d$ and $0$, respectively.
\end{lemma}
\begin{proof}
By Remark~\ref{0-in-V}, $0 \in V$, so that the affine hull of $V$ is given by  $\text{aff}(V)=\text{span}(V)$. In view of this, it suffices to show that $\overline{\text{span}}(V) = \mathbb C^d$. To that end, let $(x_i)_{i=1}^d \in \overline{\text{span}}(V)^\perp$. Let $\mf b=(b_j)_{j=1}^d$, $\mf n=(n_j)_{j=1}^d$ and the associated kernel $K_{\mf b,\mf n}$ be defined on $\mathbb H_\rho$. Since we know that $f_{\mf b, \mf n}(\mathbb H_\rho) \subseteq V$, we obtain
\beqn
\inp{f_{\mf b, \mf n}(s)}{(x_i)_{i=1}^d} \overset{\eqref{f-at-s}}= \sum_{i=1}^d \sqrt{b_i} \overline{x_i} n_i^{-s} = 0,\quad s \in \mathbb H_\rho.
\eeqn
It now follows from Proposition~\ref{uniq-dir} that $\sqrt{b_i} \overline{x_i} = 0$ for all $i$. Since $b_i$'s are positive, we conclude that $x_i= 0$ for all $i$. This completes the proof.
\end{proof}
Recall that for a variety $V$ in $\mathbb B_d$, $\mathcal M_{V}$ is the restriction of $\clm(H^2_d)$ to $V$, where $H^2_d$ is the Drury-Arveson space over $\mathbb B_d$.
\begin{proposition}
\label{never-isom-iso}
Let $d_1, d_2\in \mathbb N \cup \{\infty\}$ with $d_1\ne d_2$. Let $(\mf b, \mf n)$ and $(\mf c, \mf m)$ be two pairs of weight and frequency data of length $d_1$ and $d_2$, respectively. If $V_1$ and $V_2$ are the multiplier varieties associated with $K_{\mf b, \mf n}$ and $K_{\mf c, \mf m}$, respectively, then $\mathcal{M}_{V_1}$ and $\mathcal{M}_{V_2}$ are not isometrically isomorphic.
\end{proposition}
\begin{proof} 
Without loss of generality, assume that $d_1 < d_2$. Then $V_1$ is naturally embedded in $\mathbb B_{d_2}$, so that both the varieties $V_1$ and $V_2$ lie in $\mathbb B_{d_2}$.  
By Lemma~\ref{affine-dim}, 
    \[
    \dim(\text{aff}(V_1))=d_1\neq d_2=\dim (\text{aff}(V_2)).
    \]
so that by \cite[Lemma~4.2]{Sa-Sh}, there is no automorphism of $\mathbb B_{d_2}$ that maps $V_1$ onto $V_2$. Therefore, by \cite[Theorem~4.6 \& Theorem~4.8]{Sa-Sh}, $\mathcal M_{V_1}$ and $\mathcal M_{V_2}$ are not isometrically isomorphic. This completes the proof.
\end{proof}

By the result above, in addressing the isometric isomorphism problem for multiplier algebras, it suffices to consider CNP Dirichlet series kernels with the same frequency length.
A natural approach to this problem is to determine when the associated multiplier varieties coincide. Recall that in the case where the logarithms of the frequency data are linearly independent over $\mathbb Q$, the multiplier variety is always the full ball $\mathbb B_d$, and hence they are equal. In the dependent case, a similar conclusion holds under the notion of similar pattern. Recall that a pair of weight and frequency data $(\mf b,\mf n)$ and $(\mf c, \mf m)$ of  the same length admit a similar pattern if:  
\begin{itemize}
\item $\mathcal J(\mf n) = \mathcal J(\mf m)$, and
\item for each $J \in \mathcal J(\mf n)$, the unique partition $\{J_1, J_2\}$ and the unique $|J|$-tuple $(\beta_i)_{i \in J}$ associated with $J$ remain the same whether $J$ is viewed as a subset of $\mathcal J(\mf n)$ or $\mathcal J(\mf m)$, and the following identity holds:
\beq
\label{assump-b-and-c}
\prod\limits_{i \in J_1}
 		c_i^{\beta_i}\prod\limits_{i \in J_2}  
   b_{i}^{\beta_i}=\prod\limits_{i \in J_2}c_i^{\beta_i}\prod\limits_{i \in J_1}b_i^{\beta_i}.
\eeq
\end{itemize}

An observant reader must have noticed already that if $\mf n$ and $\mf m$ are such that the sets $\{\ln n_i: i=1,\dots,d\}$ and $\{\ln m_i: i=1\dots,d\}$ are linearly independent over $\mathbb Q$, then $\mathcal J(\mf n)$ and $\mathcal J(\mf m)$ are empty. Consequently,  by definition, for any weight data $\mf b$ and $\mf c$, the pairs $(\mf b, \mf n)$ and $(\mf c, \mf m)$ admit a similar pattern in a trivial sense. Thus, nontrivial examples arise only in the dependent case. We also note that the above notion of similarity defines an equivalence relation on the collection of pairs consisting of weight data and frequency data associated with CNP Dirichlet series kernels. To keep things in perspective, consider the following examples. 
\begin{example}
Let $d = 3$, and consider $\mf n = (2, 3, 6)$ and $\mf m = (m_1, m_2, m_3)$, where $m_3=m_1m_2$ and $\{\ln m_1, \ln m_2\}$ is linearly independent over $\mathbb Q$. Let $\mf b = (b_1,b_2,b_3)$ and $\mf c = (c_1,c_2,c_3)$ be triples of positive real numbers.
Then $\mathcal J(\mf n) = \mathcal J(\mf m) = \{\{1,2,3\}\}$, and the unique partition and triple associated with $\{1,2,3\}$ are $\{ \{1,2\}, \{3\} \}$ and $(1,1,1)$, respectively, whether $\{1,2,3\}$ is viewed as an element of $\mathcal J(\mf n)$ or $\mathcal J(\mf m)$. Then, $(\mf b, \mf n)$ and $(\mf c, \mf m)$ admit a similar pattern if and only if $b_3c_1c_2=c_3b_1b_2$. 
\end{example}
\begin{Non-example}
Consider $\mf n = (q_1, q_2, q_1^2 q_2)$ and $\mf{m} = (q_1, q_2, q_1 q_2^2)$, where $q_1, q_2\Ge 2$ and the set $\{\ln q_1, \ln q_2\}$ is linearly independent of $\mathbb Q$. In this case, $\mathcal J(\mf n) = \mathcal J(\mf m) = \{\{1,2,3\}\}$. However, the unique triple associated with $\{1,2,3\}$ as an element of $\mathcal J(\mf n)$ is $(2,1,1)$, while as an element of $\mathcal J(\mf m)$ it is $(1,2,1)$. Thus, for any $3$-tuples $\mf b$ and $\mf c$ of weight data, $(\mf b, \mf n)$ and $(\mf c, \mf m)$ do not admit a similar pattern. 
\end{Non-example}

The next result completely characterizes when two multiplier varieties coincide. 

\begin{proposition}\label{equality of varieties}
Let $d \in \ntwo \cup \{\infty\}$. Let  $(\mf b, \mf n)$ and $(\mf c, \mf m)$ be two pairs of weight and frequency data of length d. Suppose that $V_1$ and $V_2$ are the multiplier varieties associated with $K_{\mf b, \mf n}$ and $K_{\mf c, \mf m}$, respectively.
Then the following statements are equivalent.

\textup{(i)} The pairs $(\mf b, \mf n)$ and $(\mf c, \mf m)$ admit a similar pattern. 

\textup{(ii)} $V_1=V_2$.
\end{proposition}

\begin{proof}
Suppose that $\mf b=(b_j)_{j=1}^d$, $\mf n=(n_j)_{j=1}^d$, $\mf c=(c_j)_{j=1}^d$ and $\mf m=(m_j)_{j=1}^d$ are $d$-tuples, and the kernels $K_{\mf b,\mf n}$ and $K_{\mf c,\mf m}$ are defined on $\mathbb H_{\rho}$ and $\mathbb H_{\sigma}$, respectively. For the proof of $(i)\implies (ii)$, let $\mf z\in V_1$. Then $q_{_J}(\mf z)=0$, for every polynomial $q_{_J}$ associated with $J\in \mathcal J(\mf n)$, which implies that  
\beqn
 \prod\limits_{i \in J_1}
 		b_{i}^{\frac{\beta_i}{2}}\prod\limits_{i \in J_2}  
   z_i^{\beta_i}=\prod\limits_{i \in J_2}b_{i}^{\frac{\beta_i}{2}}\prod\limits_{i \in J_1}z_i^{\beta_i}, \ \text{ for all } J \in \mathcal J(\mf n).
\eeqn
Now, by the similarity assumption, we have   
\beqn
\prod\limits_{i \in J_2}c_{i}^{\frac{\beta_i}{2}}\prod\limits_{i \in J_1}z_i^{\beta_i}
= \frac{\prod\limits_{i \in J_2}
 		c_{i}^{\frac{\beta_i}{2}}\prod\limits_{i \in J_1}  
   b_{i}^{\frac{\beta_i}{2}}}{\prod\limits_{i \in J_2}b_{i}^{\frac{\beta_i}{2}}}\prod\limits_{i \in J_2}z_i^{\beta_i} \overset{\eqref{assump-b-and-c}}= 
   \prod\limits_{i \in J_1}
 		c_{i}^{\frac{\beta_i}{2}} \prod\limits_{i \in J_2}  
   z_i^{\beta_i}, \ \text{ for all } J \in \mathcal J(\mf m),
\eeqn
and hence $q_{_J}(\mf z) = 0$ for all polynomial $q_{_J}$ associated with  $J \in \mathcal J(\mf m)$. Therefore, by Theorem~\ref{variety-gen}, $\mf z \in V_2$, and thus $V_1 \subseteq V_2$. By symmetry, we also obtain $V_2 \subseteq V_1$, and hence $V_1 = V_2$.

\vspace{0.3em}
\noindent $(ii)\implies (i):$ By the definition of multiplier variety, $f_{\mf b, \mf n}(\mathbb H_{\rho})\subseteq V_1= V_2$. This implies that for every polynomial $q_{_J}$ associated with $J\in \mathcal J(\mf m)$, we have 
\[q_{_J}(f_{\mf b, \mf n}(s))= q_{_J}((\sqrt{b_i}n_i^{-s})_{i=1}^d)=0\ \text{ for all } s\in\mathbb H_{\rho}.\] 
If $\{J_1,J_2\}$ and $(\beta_i)_{i \in J}$ with $\text{gcd}(\beta_i)_{i \in J}=1$ are the unique partition and the unique tuple associated with $J$, then the expression of the polynomial $q_{_J}$ yields:
\[
\big(\prod_{i\in J_1}c_i^{\frac{\beta_i}{2}}\prod_{i\in J_2}b_i^{\frac{\beta_i}{2}}\big) \prod_{i\in J_2}(n_i^{\beta_i})^{-s}- \big(\prod_{i\in J_1}b_i^{\frac{\beta_i}{2}}\prod_{i\in J_2}c_i^{\frac{\beta_i}{2}}\big)\prod_{i\in J_1}(n_i^{\beta_i})^{-s}=0\ \text{for all } s\in \mathbb H_{\rho}.
\]
Since $b_i>0$ and $c_i>0$ for all $i$, the uniqueness of Dirichlet series implies that this identity holds if and only if  
\beq
\label{generating identity for n}
\prod_{i\in J_1}c_i^{\beta_i}\prod_{i\in J_2}b_i^{\beta_i}=\prod_{i\in J_1}b_i^{\beta_i}\prod_{i\in J_2}c_i^{\beta_i}\ \text{ and }
\prod_{i\in J_2}n_i^{\beta_i}= \prod_{i\in J_1}n_i^{\beta_i}.
\eeq
Thus, for each $J\in \mathcal J(\mf m)$, we obtain relations of the form ~\eqref{generating identity for n} determined by the unique partition and the $|J|$-tuple associated with $J$. By symmetry, considering $f_{\mf c, \mf m}(\mathbb H_{\sigma})\subseteq V_1$, one similarly derives that for each $J'\in \mathcal J(\mf n)$, with associated unique partition $\{J_1', J_2'\}$ and unique tuple $(\beta_i')_{i=1}^{|J'|}$ satisfying $\text{gcd}(\beta_i')_{i=1}^{|J'|}=1$, we have   
\beq
\label{generating identity for m}
\prod_{i\in J_1'}c_i^{\beta_i'}\prod_{i\in J_2'}b_i^{\beta_i'}=\prod_{i\in J_1'}b_i^{\beta_i'}\prod_{i\in J_2'}c_i^{\beta_i'}\ \text{ and }
\prod_{i\in J_2'}m_i^{\beta_i'}= \prod_{i\in J_1'}m_i^{\beta_i'}.
\eeq

To conclude $(\mf b, \mf n)$ and $(\mf c, \mf m)$ admit a similar pattern, let $J\in \mathcal J(\mf m)$ with associated unique partition $\{J_1,J_2\}$ and unique tuple $(\beta_i)_{i \in J}$ satisfying $\text{gcd}(\beta_i)_{i \in J}=1$. Then, by ~\eqref{generating identity for n}, the set $\{\ln n_i: i\in J\}$ is linearly dependent over $\mathbb Q$. Thus, there exists $\tilde{J}\subseteq J$ such that $\tilde{J}\in \mathcal J(\mf n)$. We claim that $\tilde{J}= J$. Suppose not, that is, $\tilde{J}\subsetneq J$. Since $\tilde{J}\in \mathcal J(\mf n)$, by ~\eqref{generating identity for m}, the set
$\{\ln m_i: i\in \tilde{J}\}$ is linearly dependent over $\mathbb Q$, which is a contradiction as $J\in \mathcal J(\mf m)$ and $\tilde{J}$ is a proper subset of $J$. Therefore, $J=\tilde{J}$, and hence $J \in \mathcal J(\mf n)$, proving that $\mathcal J(\mf m)\subseteq \mathcal J(\mf n)$. Moreover, by Proposition~\ref{uniq-decom} and ~\eqref{generating identity for n}, the unique partition and the tuple associated with $J$ remain unchanged when $J$ is regarded as a member of either $\mathcal J(\mf m)$ or $\mathcal J(\mf n)$, and the identity \eqref{assump-b-and-c} holds. 
By symmetry, we also obtain $\mathcal J(\mf n)\subseteq \mathcal J(\mf m)$, and thus $\mathcal J(\mf n)=\mathcal J(\mf m)$. Therefore, $(\mf b, \mf n)$ and $(\mf c, \mf m)$ admit a similar pattern. This completes the proof. 
\end{proof}

As an immediate consequence of the preceding result, we obtain the following characterization of isometric isomorphisms between multiplier algebras of CNP Dirichlet series kernels.   
For the multiplier variety $V$ associated with $K_{\mf b, \mf n}$, recall that 
the composition map $C_{f_{\mf b,\mf n}}$, induced by $f_{\mf b, \mf n}$, 
 \beq 
\label{Composition}
 C_{f_{\mf b, \mf n}}: \mathcal M_V \to \mathcal M(\clh(K_{\mf b, \mf n})),\  C_{f_{\mf b, \mf n}}(h) = h \circ f_{\mf b, \mf n},\quad h \in \mathcal M_V
\eeq
is an isometric isomorphism (refer to \cite[section 5]{MS} for more details). 
 
\begin{theorem}\label{var-equal-set} 
Let $d \in \ntwo \cup \{\infty\}$. Let  $(\mf b, \mf n)$ and $(\mf c, \mf m)$ be two pairs of weight and frequency data of length d. Suppose that $V_1$ and $V_2$ are the multiplier varieties associated with $K_{\mf b, \mf n}$ and $K_{\mf c, \mf m}$, respectively.
Then the following statements are equivalent.

\textup{(i)} The pairs $(\mf b, \mf n)$ and $(\mf c, \mf m)$ admit a similar pattern.

\textup{(ii)} The identity map $\mathrm{Id}: \mathcal M_{V_1}\to \mathcal M_{V_2}$ is an isometric isomorphism.

\textup{(iii)} There exists an isometric isomorphism $\Phi: \mathcal M (\mathcal H(K_{\mf b, \mf n}))\to \mathcal M(\mathcal H(K_{\mf c, \mf m}))$ such that 
\[C_{f_{\mf c, \mf m}}^{-1}\circ \Phi\circ C_{f_{\mf b, \mf n}}= \mathrm{Id}: \mathcal M_{V_1}\to \mathcal M_{V_2},\] 
where $C_{f_{\mf b, \mf n}}$ is as in ~\eqref{Composition}.
\end{theorem}
\begin{proof}
The equivalence of (i) and (ii) follows directly from Proposition~\ref{equality of varieties}. The equivalence of \textup{(ii)} and \textup{(iii)} follows from the commutativity of the diagram below:
\begin{center}
\begin{tikzcd}
\mathcal{M}_{V_1} \arrow[r, "\mathrm{Id}"] \arrow[d, "C_{f_{\mf b, \mf n}}"]
&\mathcal M_{V_2}  \arrow[d, "C_{f_{\mf c, \mf m}}"] \\
\mathcal M(\clh(K_{\mf b, \mf n})) \arrow[r, "\Phi"]
& \mathcal M(\clh(K_{\mf c, \mf m}))
\end{tikzcd}
\end{center}
where all the maps are isometric isomorphisms. 
\end{proof}

We now present a family of examples that illustrate Theorem~\ref{var-equal-set}. In each case, we conclude--by virtue of Theorem~\ref{var-equal-set}--that the corresponding multiplier algebras are isometrically isomorphic. 
\begin{example}
\label{iso-isom-example}
(i) Let $(\alpha_1, \alpha_2) \in \mathbb N^2$, and define $\mf n = (2, 3, 2^{\alpha_1}3^{\alpha_2})$ and $\mf m = (q_1, q_2, q_1^{\alpha_1}q_2^{\alpha_2})$, where the set $\{\ln q_1, \ln q_2\}$ is linearly independent over $\mathbb Q$. Then $\mathcal{J}(\mf n) = \mathcal{J}(\mf m) = \{\{1,2,3\}\}$. The unique partition and the unique tuple associated with $\{1,2,3\}$, when viewed as an element of either $\mathcal{J}(\mf n)$ or $\mathcal{J}(\mf m)$, are $\{ \{1,2\}, \{3\}\}$ and $(\alpha_1, \alpha_2, 1)$, respectively. Let $\mf b=(b_1,b_2,b_3)$ and $\mf c=(c_1,c_2,c_3)$ be triples of positive real numbers. Then $(\mf b,\mf n)$ and $(\mf c,\mf m)$ admit a similar pattern if and only if 
\[b_3c_1^{\alpha_1}c_2^{\alpha_2} = c_3 b_1^{\alpha_1}b_2^{\alpha_2}.\] 
Therefore, by Theorem~\ref{var-equal-set}, if the above identity holds, then the multiplier algebras $\mathcal M(\mathcal H(K_{\mf b, \mf n}))$ and $\mathcal M(\mathcal H(K_{\mf c, \mf m}))$ are isometrically isomorphic.


(ii) Let $d\Ge 4$ be an integer or $d= \infty$, and let $\{p_n\}_{n=1}^{\infty}$ be the enumeration of primes in increasing order. Define \[\mf n=(n_j)_{j=1}^d = (2, 3, 4, p_4, p_5,\dots p_d) \text{ and } \mf m =(m_j)_{j=1}^d= (q_1, q_2, q_1^2, p_{n_0}, p_{n_0+1},\dots, p_{d+n_0-4}),\] where $p_{n_0} > \text{max}\{q_1, q_2\}$ and $\{\ln q_1, \ln q_2\}$ is a linearly independent set over $\mathbb Q$. In this case, it is easy to see that $\mathcal J(\mf n) = \mathcal J(\mf m) = \{\{1,3\}\}$, and the unique pair associated with $\{1,3\}$ when viewed as an element in $\mathcal J(\mf n)$ or $\mathcal J(\mf m)$ is $(2,1)$. Let $\mf b = (b_j)_{j=1}^d$ and $\mf c = (c_j)_{j=1}^d$ be $d$-tuples of positive real numbers. If $d=\infty$, assume that $\sigma_a(\sum_j b_jn_j^{-s})<\infty$ and $\sigma_a(\sum_j c_jm_j^{-s})<\infty$. Then observe that $(\mf b, \mf n)$ and $(\mf c, \mf m)$ admit a similar pattern if and only if $b_1^2c_3 = c_1^2b_3$. Therefore, by Theorem~\ref{var-equal-set}, if $b_1^2c_3 = c_1^2b_3$, then $\mathcal M(\mathcal H(K_{\mf b, \mf n}))$ and $\mathcal M(\mathcal H(K_{\mf c, \mf m}))$ are isometrically isomorphic.

(iii) Consider $\mf n = (6, 10, 21, 35, 360)$ and $\mf m = (q_1, q_2, q_3, q_4, q_1^2q_2)$ such that $q_1q_4 = q_2q_3$, and every proper subset of $\{\ln q_i: i=1,\dots, 4\}$ is linearly independent over $\mathbb Q$. Let $\mf b=(b_j)_{j=1}^4$ and $\mf c=(c_j)_{j=1}^4$ be tuples of positive real numbers. Then 
\[\mathcal J(\mf n) = \mathcal J(\mf m) = \{\{1,2,3,4\}, \{1,2,5\}\}.\]
The unique partition and the unique tuple associated with $\{1,2,3,4\}$, when viewed as an element in $\mathcal J(\mf n)$ or $\mathcal J(\mf m)$, are $\{\{1,4\}, \{2,3\}\}$ and $(1,1,1,1)$, respectively. Similarly, the unique partition and the unique tuple associated with $\{1,2,5\}$, when viewed as an element in $\mathcal J(\mf n)$ or $\mathcal J(\mf m)$, are $\{\{1,2\}, \{5\}\}$ and $(2,1,1)$, respectively. Thus, the pair $(\mf b,\mf n)$ and $(\mf c,\mf m)$ admit a similar pattern if and only if 
\[b_2b_3c_1c_4 = c_2c_3 b_1b_4\ \text{ and }\ b_1^2b_2c_5 = c_1^2c_2b_5.\] 
Therefore, if $\mf b$ and $\mf c$ are chosen to satisfy these identities, then by Theorem~\ref{var-equal-set}, $\mathcal M(\mathcal H(K_{\mf b, \mf n}))$ and $\mathcal M(\mathcal H(K_{\mf c, \mf m}))$ are isometrically isomorphic.
\end{example}

\section{CNP kernels with generating frequencies}\label{6}

In this section, we consider a specific class of CNP Dirichlet series kernels--determined by particular choices of weight and frequency data--and study the isomorphism problem for their multiplier algebras. One of our objectives is to provide a negative answer to Question 2, as stated in the Introduction. We now proceed to establish this result. 

\begin{proposition}
\label{negative}
Let $3\Le d\in \mathbb N$, and let $(\mf b, \mf n=(n_j)_{j=1}^d)$ be a pair of weight and frequency data of length $d$. Suppose that the set $\{\ln n_{i}: i=1,\dots, d-1\}$ is linearly independent over $\mathbb Q$, and there exist  $\emptyset\neq I\subseteq\{1,\ldots, d-1\}$ and a tuple $(\alpha_i)_{i\in I}$ of natural numbers such that
\beqn
n_d=\prod_{i\in I}n_i^{\alpha_i}.
\eeqn
Then $\mathcal M(\mathcal H(K_{\mf b, \mf n}))$ is not isomorphic to $\mathcal M_{d-1}$. 
\end{proposition}
\begin{proof}
We begin by determining the multiplier variety $V$ associated with $K_{\mf b, \mf n}$.
Under the assumptions on $\mf n$, we have $\mathcal J(\mf n) = \{I\cup\{d\}\}$. The unique partition and the unique tuple associated with $I\cup\{d\}$ are $\{I, \{d\} \}$ and $(\alpha_i)_{i\in I\cup\{d\}}$, respectively, with $\alpha_d=1$. Let $\mf b=(b_j)_{j=1}^d$, then by Theorem~\ref{variety-gen}, 
\beqn
    V &=& \mathbb B_d \cap Z\Big(\prod_{i\in I} b_i^{\frac{\alpha_i}{2}} z_d - \sqrt{b_d}\prod_{i\in I} z_i^{\alpha_i}\Big)\\
    &=& \mathbb B_d \cap \Big\{(z_1, z_2,\ldots, z_{d-1}, \sqrt{b_d} \prod_{i\in I} b_i^{-\frac{\alpha_i}{2}}z_i^{\alpha_i}): (z_1,z_2,\dots, z_{d-1})\in \mathbb C^{d-1} \Big\}.
\eeqn
We claim that the variety $V$ is not biholomorphic to $\mathbb B_{d-1}$. To establish this, define 
\beqn
    \tilde{V}:=\Big\{(z_1,z_2, \dots, z_{d-1}) \in\mathbb B_{d-1} \ : \ \sum_{i=1}^{d-1}|z_i|^2+ b_d \prod_{i\in I}b_i^{-\alpha_i} |z_i|^{2\alpha_i} < 1\Big\}.
\eeqn
Observe that $V$ can be realized as a graph over the domain $\tilde{V}\subset \mathbb C^{d-1}$. Consequently, $\tilde{V}$ and $V$ are biholomorphic through the map $\varphi:\Tilde{V}\to V$ defined by 
\beqn
    \varphi(z_1,z_2,\dots, z_{d-1})=\Big(z_1,z_2,\ldots,z_{d-1}, \sqrt{b_d} \prod_{i\in I}b_i^{-\frac{\alpha_i}{2}}z_i^{\alpha_i}\Big).
\eeqn
Therefore, to prove the claim, it suffices to show that $\tilde{V}$ and $\mathbb B_{d-1}$ are not biholomorphically equivalent. Observe that both $\Tilde{V}$ and $\mathbb B_{d-1}$ are bounded Reinhardt domains in $\mathbb C^{d-1}$. By \cite[Main Theorem]{Sunada}, $\tilde{V}$ and $\mathbb B_{d-1}$ are biholomorphically equivalent if and only if there exists a linear transformation of the form 
\[\psi:\mathbb C^{d-1}\rar \mathbb C^{d-1},\quad  \psi(z_1,z_2,\dots, z_{d-1}) = (r_1 z_{\sigma(1)},r_2 z_{\sigma(2)}, \dots, r_{d-1}z_{\sigma(d-1)}),\] 
 for some permutation $\sigma$ of $\{1,\ldots, d-1\}$ and constants $r_i>0$,  such that $\psi(\tilde{V} )= \mathbb B_{d-1}$. Suppose, for a contradiction, that such a linear map $\psi$ exists. Let $\{e_i:i=1, \dots,d-1\}$ be the standard basis of $\mathbb C^{d-1}$. 
 Note that the boundary $\partial\tilde{V}$ of $\tilde{V}$
is given by 
\[
   \partial\tilde{V}:=\Big\{(z_1,z_2,\ldots,z_{d-1})\in\mathbb C^{d-1}:\sum_{i=1}^{d-1}|z_i|^2+b_d\prod_{i\in I}b_i^{-\alpha_i}|z_i|^{2\alpha_i} = 1\Big\}.
   \]
Fix $i_0 \in I$. Then, $e_i\in \partial \tilde{V}$ for all $i \in  \{1,\ldots, d-1\} \setminus \{i_0\}$, so that
  $r_{\sigma^{-1}(i)} e_{\sigma^{-1}(i)} = \psi(ze_i)\in \partial \mathbb B_{d-1}$ for all $i \neq i_0$,
 and hence 
 \beq
 \label{ri=1-ppn}
 r_{\sigma^{-1}(i)} = 1, \quad i \neq i_0.
 \eeq
Since $\psi(\partial \tilde{V}) = \partial \mathbb B_{d-1}$, it follows that 
   \beqn
   \sum_{i=1}^{d-1} |z_i|^2 + b_d\prod_{i \in I}b_i^{-\alpha_i}|z_i|^{2\alpha_i}  = \sum_{i=1}^d r_i^2 |z_{\sigma(i)}|^2, \quad (z_1, z_2,\ldots z_{d-1}) \in \partial \tilde{V},
   \eeqn
 and hence, by \eqref{ri=1-ppn}, 
  \beqn
  |z_{i_0}|^2 + b_d\prod_{i \in I}b_i^{-\alpha_i}|z_i|^{2\alpha_i}  =  r_{\sigma^{-1}(i_0)}^2 |z_{i_0}|^2, \quad (z_1,z_2, \ldots z_{d-1}) \in \partial \tilde{V}.
  \eeqn
Since there exists a point $(z_1, \ldots, z_{d-1}) \in \partial \tilde{V}$ with $z_{i_0} \neq 0$, we get
 \beq
 \label{reduced-eqn}
  1 + b_db_{i_0}^{-\alpha_{i_0}}|z_{i_0}|^{2(\alpha_{i_0}-1)}\prod_{i \in I \setminus \{i_0\}}b_i^{-\alpha_i}|z_i|^{2\alpha_i}  =  r_{\sigma^{-1}(i_0)}^2, \quad 
  \eeq
for all $(z_1, \ldots z_{d-1}) \in \partial \tilde{V}$ with $z_{i_0} \neq 0$. Now, we show that 
\beq\label{r=1_for remaining}
r_{\sigma^{-1}(i_0)} = 1. 
\eeq
Suppose $I = \{i_0\}$. Since $\mf n$ is a tuple of distinct integers, $\alpha_{i_0} \Ge 2$. Then, letting $|z_{i_0}| \to 0$ in \eqref{reduced-eqn} yields that $r_{\sigma^{-1}(i_0)} = 1$. In the other case, if $|I| \Ge 2$, then there exists $1\Le j\Le d-1$ such that $j\in I \setminus \{i_0\}$. Applying \eqref{reduced-eqn} for a point $(z_1,z_2,\ldots,z_{d-1}) \in \partial \tilde{V}_1$ with $z_j = 0$  and $z_{i_0}\neq 0$ yields that $r_{\sigma^{-1}(i_0)} = 1$. This proves \eqref{r=1_for remaining}.
Therefore, by \eqref{ri=1-ppn} and \eqref{r=1_for remaining}, $r_i=1$ for all $i=1, \ldots, d-1$, and so 
  \beqn
  \psi(z_1,z_2,\dots, z_{d-1}) = (z_{\sigma(1)},z_{\sigma(2)},\dots, z_{\sigma(d-1)}), \quad (z_1,z_2,\dots, z_{d-1}) \in \mathbb C^{d-1}.
  \eeqn
This in particular implies that 
\beqn
  \psi^{-1}(\mathbb B_{d-1}) = \mathbb B_{d-1} = \tilde{V},
  \eeqn
which is false since $\tilde{V} \subsetneq \mathbb B_{d-1}$. Hence, $\tilde{V}$ is not biholomorphic to $\mathbb B_{d-1}$, and consequently, neither is $V$, as claimed.

 Suppose now, for a contradiction, that $\clm_V\cong\clm_{d-1}$. Since both $V$ and $\mathbb{B}_{d-1}$ are irreducible varieties, it follows from ~\cite[Theorem 5.6]{DRS} that $V$ and $\mathbb B_{d-1}$ must be biholomorphic, which we just ruled out. Therefore, $\clm_V \ncong \clm_{d-1}$. Moreover, since $\mathcal M_V \cong \mathcal M(\mathcal H(K_{\mf b, \mf n}))$, we conclude that 
 \[\mathcal M(\mathcal H(K_{\mf b, \mf n}))\ncong \mathcal M_{d-1}.\] 
 This completes the proof.  
\end{proof}

For the class of kernels considered in the above theorem, we obtain a complete classification of multiplier algebras up to isometric and algebraic isomorphism. Remarkably, within this setting, every algebraic isomorphism between multiplier algebras is automatically isometric, and such an isomorphism occurs precisely when the associated multiplier varieties coincide up to a permutation automorphism.  

\begin{theorem}\label{isometric isomorphism}
 Let $3\Le d\in \mathbb N$, and let $(\mf b, \mf n=(n_j)_{j=1}^d)$ and $(\mf c, \mf m=(m_j)_{j=1}^d)$ be two pairs of weight and frequency data of length d. Suppose the sets $\{\ln n_{i}: i=1,\dots,d-1\}$ and $\{ \ln m_i: i=1,\dots, d-1\}$ are linearly independent over $\mathbb Q$, and there exist non-empty subsets $I_1$, $I_2$ of $\{1,\ldots,d-1\}$ and tuples $(\alpha_i)_{i\in I_1}$, $(\beta_i)_{i\in I_2}$ of natural numbers such that
\beq\label{n-m-decomp}
n_d=\prod_{i\in I_1}n_i^{\alpha_i}\ \text{ and }\ m_d=\prod_{i\in I_2}m_i^{\beta_i}.
\eeq
 Then the following are equivalent:
\begin{itemize}
\item[(i)] $\mathcal{M}(\mathcal{H}(K_{\mf{b}, \mf{n}}))$ and $\mathcal{M}(\mathcal{H}(K_{\mf{c}, \mf{m}}))$ are isomorphic,
\item[(ii)] there exists a permutation $\sigma$ of the index set $\{1,\dots,d\}$ such that the pair $(\sigma(\mf b), \sigma(\mf n))$ and $(\mf c, \mf m)$ admit a similar pattern,
\item[(iii)] $\mathcal M(\mathcal H(K_{\mf b, \mf n}))$ and $\mathcal M(\mathcal H(K_{\mf c, \mf m}))$ are isometrically isomorphic. 
\end{itemize}  
\end{theorem}
\begin{proof}
Because $K_{\mf b, \mf n} = K_{\sigma(\mf b), \sigma(\mf n)}$ for any permutation $\sigma$ of the index set $\{1,\dots,d\}$, the implication $\textup{(ii)} \implies \textup{(iii)}$ follows from Theorem~\ref{var-equal-set}. The implication $\textup{(iii)} \implies \textup{(i)}$ is trivial. 

To complete the proof, we show that $\textup{(i)} \implies \textup{(ii)}$. To this end, let $V_1$ and $V_2$ be the multiplier varieties associated with $K_{\mf b, \mf n}$ and $K_{\mf c, \mf m}$, respectively. Since the algebras $\clm(\clh(K_{\mf b,\mf n}))$ and $\clm(\clh(K_{\mf c,\mf m}))$ are isomorphic, $\clm_{V_1}$ is isomorphic to $\clm_{V_2}$. Let $\mf b=(b_j)_{j=1}^d$ and $\mf c=(c_j)_{j=1}^d$, then by the assumptions on $\mf n$ and $\mf m$, as given in \eqref{n-m-decomp}, and by Theorem~\ref{variety-gen}, we have 
   \[
   V_1=\Big\{(z_1,z_2,\ldots,z_d)\in\mathbb B_d: z_d=\sqrt{b_d}\prod_{i\in I_1}b_i^{-\frac{\alpha_i}{2}}z_i^{\alpha_i}\Big\}
   \]
   and
   \[
   V_2=\Big\{(z_1,z_2,\ldots,z_d)\in\mathbb B_d: z_d=\sqrt{c_d}\prod_{i\in I_2}c_i^{-\frac{\beta_i}{2}}z_i^{\beta_i}\Big\}.
   \]
Since $V_1$ and $V_2$ are irreducible varieties and $\mathcal M_{V_1}\cong \mathcal M_{V_2}$, by \cite[Theorem 5.6]{DRS}, $V_1$ and $V_2$ are biholomorphically equivalent.  
Now, define the domains
   \[
   \tilde{V}_1:=\Big\{(z_1,z_2,\ldots,z_{d-1})\in\mathbb C^{d-1}:\sum_{i=1}^{d-1}|z_i|^2+b_d\prod_{i\in I_1}b_i^{-\alpha_i}|z_i|^{2\alpha_i}<1\Big\}
   \]
   and
   \[
   \tilde{V}_2:=\Big\{(z_1,z_2,\ldots,z_{d-1})\in\mathbb C^{d-1}:\sum_{i=1}^{d-1}|z_i|^2+c_d\prod_{i\in I_2}c_i^{-\beta_i}|z_i|^{2\beta_i}<1\Big\}.
   \]
Observe that $V_1$ and $V_2$ are graphs over $\tilde{V}_1$ and $\tilde{V}_2$, respectively. Hence, $V_1$ is biholomorphic to $\tilde{V}_1$ via the map $\varphi_1:\tilde{V}_1\to V_1$ defined by 
   \[
   \varphi_1(z_1,z_2,\ldots,z_{d-1})=\Big(z_1,z_2,\ldots,z_{d-1},\sqrt{b_d}\prod_{i\in I_1}b_i^{-\frac{\alpha_i}{2}}z_i^{\alpha_i}\Big),\quad (z_1,z_2,\ldots,z_{d-1})\in\tilde{V}_1,
   \]
 and  $V_2$ is also biholomorphic to $\tilde{V_2}$ via 
  \[
    \varphi_2(z_1,z_2,\ldots,z_{d-1})=\Big(z_1,z_2,\ldots,z_{d-1},\sqrt{c_d}\prod_{i\in I_2}c_i^{-\frac{\beta_i}{2}}z_i^{\beta_i}\Big), \quad (z_1,z_2,\ldots,z_{d-1})\in\tilde{V}_2.
   \]
Therefore, $\tilde{V}_1$ and $\tilde{V}_2$ are biholomorphic. Since $\tilde{V}_1$ and $\tilde{V}_2$ are bounded Reinhardt domains in $\mathbb C^{d-1}$, by \cite[Main Theorem]{Sunada}, there exists a transformation 
\[\psi:\mathbb C^{d-1}\to\mathbb C^{d-1}, \quad\psi(z_1,z_2,\dots, z_{d-1})=(r_1z_{\sigma(1)}, r_2z_{\sigma(2)},\dots, r_{d-1}z_{\sigma(d-1)}),\] where $r_i>0$.
The boundaries $\partial\tilde{V_1}$ and $\partial\tilde{V_2}$ of $\tilde{V_1}$ and $\tilde{V_2}$
are given by  \[
   \partial\tilde{V_1}:=\Big\{(z_1,z_2,\ldots,z_{d-1})\in\mathbb C^{d-1}:\sum_{i=1}^{d-1}|z_i|^2+b_d\prod_{i\in I_1}b_i^{-\alpha_i}|z_i|^{2\alpha_i} = 1\Big\}
   \]
   and
   \[
   \partial\tilde{V_2}:=\Big\{(z_1,z_2,\ldots,z_{d-1})\in\mathbb C^{d-1}:\sum_{i=1}^{d-1}|z_i|^2+c_d\prod_{i\in I_2}c_i^{-\beta_i}|z_i|^{2\beta_i} = 1\Big\}.
   \]
 Then, 
 since $\psi(\partial\tilde{V_1})=\partial\tilde{V_2}$, it follows that for all  $(z_1,z_2,\ldots,z_{d-1}) \in\partial\tilde{V_1}$,
\beq\label{dz_i_expression}
 \sum_{i=1}^{d-1} |z_i|^2 + b_d\prod_{i\in I_1} b_i^{-\alpha_i}|z_i|^{2\alpha_i}= \sum_{i=1}^{d-1} r_i^2 |z_{\sigma(i)}|^2 + c_d\prod_{i\in I_2}c_i^{-\beta_i}r_i^{2\beta_i}|z_{\sigma(i)}|^{2\beta_i}.
   \eeq
We denote by $\{e_i: i=1,\dots, d-1\}$ the standard basis of $\mathbb C^{d-1}$. In what follows, we shall need the following lemma. 
\begin{lemma}\label{L:cardinality}
If the cardinality of any one of the sets $I_1$ and $I_2$ is greater than $1$, then so is the other. 
\end{lemma}
\begin{proof}
Without loss of generality, we assume that $|I_2|\Ge 2$. Suppose in contrast that $I_1 = \{i_0\}$ for some $i_0\in\{1,\ldots,d-1\}$. In such a case, note that $\alpha_{i_0} \Ge 2$. Since $e_i \in \partial \tilde{V}_1$ for all $i \neq i_0$, then $\psi(e_i) = r_{\sigma^{-1}(i)} e_{\sigma^{-1}(i)} \in \partial \tilde{V}_2$. It then follows from the defining identity of $\partial \tilde{V}_2$ and the assumption $|I_2|\Ge 2$ that
\beq
\label{r'i-1-almost}
r_{\sigma^{-1}(i)} = 1, \quad i \neq i_0.
\eeq
This reduces \eqref{dz_i_expression} to
\beq\label{id_gencase_case3}
|z_{i_0}|^2+b_db_{i_0}^{-\alpha_{i_0}}|z_{i_0}|^{2\alpha_{i_0}}=r_{\sigma^{-1}(i_0)}^2|z_{i_0}|^2+c_d \prod_{i\in I_2}c_i^{-\beta_i}r_{i}^{2\beta_i}|z_{\sigma(i)}|^{2\beta_i}, \quad (z_1,z_2,\ldots,z_{d-1}) \in \partial \tilde{V}_1.
\eeq
 Suppose $\sigma^{-1}(i_0) \notin I_2$. Since there exists a point $(z_1,z_2,\ldots,z_{d-1})\in\partial\tilde{V_1}$ such that $z_{i_0}=0$ and $z_i\neq 0$ for all $i\neq i_0$, we get by \eqref{id_gencase_case3} that
\beqn
c_d\prod_{i\in I_2}c_i^{-\beta_i}|z_{\sigma(i)}|^{2\beta_i}=0,
\eeqn
which is a contradiction as $z_{\sigma(i)} \neq 0$ for all $i \in I_2$. Thus, $\sigma^{-1}(i_0) \in I_2$. Now for all points $(z_1,z_2,\ldots,z_{d-1})\in\partial\tilde{V_1}$ such that $z_{i_0}\neq0$, we get by the identities \eqref{id_gencase_case3} and \eqref{r'i-1-almost} that
\beq \label{simplified-form}
1&+&b_db_{i_0}^{-\alpha_{i_0}}|z_{i_0}|^{2(\alpha_{i_0}-1)}\notag\\
&=&r_{\sigma^{-1}(i_0)}^2
 +c_dc_{\sigma^{-1}(i_0)}^{-\beta_{\sigma^{-1}(i_0)}}r_{\sigma^{-1}(i_0)}^{2\beta_{\sigma^{-1}(i_0)}}|z_{i_0}|^{2(\beta_{\sigma^{-1}(i_0)}-1)}\prod_{i\in I_2\setminus \{\sigma^{-1}(i_0)\}}c_i^{-\beta_i}|z_{\sigma(i)}|^{2\beta_i}.
\eeq 
Letting $|z_{i_0}|\rar 0$ in the above, we conclude that $r_{\sigma^{-1}(i_0)} \Le 1$. On the other hand, since $|I_2|\Ge 2$, there exists $1\Le j\Le d-1$ such that $j\neq i_0$ and $\sigma^{-1}(j) \in I_2$.   Then applying \eqref{simplified-form} for a point $(z_1,z_2,\ldots,z_{d-1}) \in \partial \tilde{V}_1$ with $z_j = 0$  and $z_{i_0}\neq 0$, we get
\beqn
 1+b_db_{i_0}^{-\alpha_{i_0}}|z_{i_0}|^{2(\alpha_{i_0}-1)}=r_{\sigma^{-1}(i_0)}^2,
 \eeqn
which suggests that $r_{\sigma^{-1}(i_0)} > 1$. This is a contradiction. This completes the proof of the lemma.
\end{proof}

We now make the claim that:

\noindent\textbf{Claim: }
\text{$r_i = 1$ for all $i =1, \ldots, d-1$, and $\sigma(I_2) = I_1$}. 

We prove this claim in the following two cases.

\noindent\textbf{Case I.} Suppose $I_2= \{j_0\}$. Then by Lemma~\ref{L:cardinality}, $I_1=\{i_0\}$ for some $i_0\in\{1,\ldots,d-1\}$. Since $\mf n$ and $\mf m$ are tuples of distinct integers, it follows that $\alpha_{i_0}, \beta_{j_0}\Ge 2$. By an argument similar to that used in the proof of the above lemma, note that  
$e_i\in\partial\tilde{V}_1$ for all $i\neq i_0$ and consequently $\psi(e_i)=r_{\sigma^{-1}(i)}e_{\sigma^{-1}(i)}\in\partial\tilde{V}_2$. Hence, 
\[r_{\sigma^{-1}(i)}=1\ \text{ for all } i\neq i_0, \sigma(j_0).\]
Suppose that $\sigma(j_0) \neq i_0$. Since $\psi(\partial \tilde{V_1}) = \partial \tilde{V_2}$, any point $(z_1, \ldots, z_{d-1}) \in \partial \tilde{V_1}$ which satisfies 
\begin{equation}\label{boundary}
|z_{i_0}|^2+|z_{\sigma(j_0)}|^2+b_db_{i_0}^{-\alpha_{i_0}}|z_{i_0}|^{2\alpha_{i_0}}=1\end{equation}
also satisfies 
\[ 
r_{\sigma^{-1}(i_0)}^2|z_{i_0}|^2+r_{j_0}^2|z_{\sigma(j_0)}|^2+c_dc_{j_0}^{-\beta_{j_0}}r_{j_0}^{2\beta_{j_0}}|z_{\sigma(j_0)}|^{2\beta_{j_0}}=1.
\]
Now, solving for $|z_{\sigma(j_0)}|$, we get $z_{i_0}$ which satisfies \eqref{boundary} also satisfies 
\beqn
r_{\sigma^{-1}(i_0)}^2|z_{i_0}|^2&+&r_{j_0}^2(1-|z_{i_0}|^2-b_db_{i_0}^{-\alpha_{i_0}}|z_{i_0}|^{2\alpha_{i_0}})\notag\\
&+&c_dc_{j_0}^{-\beta_{j_0}}r_{j_0}^{2\beta_{j_0}}(1-|z_{i_0}|^2-b_db_{i_0}^{-\alpha_{i_0}}|z_{i_0}|^{2\alpha_{i_0}})^{\beta_{j_0}}-1=0.
\eeqn
Since there are infinitely many choices of $|z_{i_0}|$ in \eqref{boundary}, the polynomial 
\beqn
r_{\sigma^{-1}(i_0)}^2x^2+r_{j_0}^2(1-x^2-b_db_{i_0}^{-\alpha_{i_0}}x^{2\alpha_{i_0}})+c_dc_{j_0}^{-\beta_{j_0}}r_{j_0}^{2\beta_{j_0}}(1-x^2-b_db_{i_0}^{-\alpha_{i_0}}x^{2\alpha_{i_0}})^{\beta_{j_0}}-1
\eeqn
vanishes identically. This is a contradiction, as the coefficient of the highest degree term 
of the polynomial is $c_dc_{j_0}^{-\beta_{j_0}}r_{j_0}^{2\beta_{j_0}}(-b_db_{i_0}^{-\alpha_{i_0}})^{\beta_{j_0}}$ and which is non-zero.
Therefore, $\sigma(j_0)=i_0$. Consequently, $r_{i}=1$ for all $i\neq j_0$. By \eqref{dz_i_expression}, we now have
 \beqn
|z_{i_0}|^2+b_db_{i_0}^{-\alpha_{i_0}}|z_{i_0}|^{2\alpha_{i_0}}=r_{j_0}^2|z_{i_0}|^2+c_dc_{j_0}^{-\beta_{j_0}}r_{j_0}^{2\beta_{j_0}}|z_{i_0}|^{2\beta_{j_0}}, \quad (z_1,z_2,\ldots,z_{d-1})\in\partial\tilde{V}_1.
\eeqn
Considering points $(z_1,z_2,\ldots,z_{d-1})\in\partial\tilde{V}_1$ with $z_{i_0}\neq0$ in the above equation we have 
\beqn
1+b_db_{i_0}^{-\alpha_{i_0}}|z_{i_0}|^{2(\alpha_{i_0}-1)}=r_{j_0}^2+c_dc_{j_0}^{-\beta_{j_0}}r_{j_0}^{2\beta_{j_0}}|z_{i_0}|^{2(\beta_{j_0}-1)}.
\eeqn
Since $\alpha_{i_0}, \beta_{j_0} \Ge 2$, letting $|z_{i_0}|\to 0$ finally establishes that $r_{j_0}=1$.

\noindent\textbf{Case II.} Suppose $|I_2|\Ge 2$. Then by Lemma~\ref{L:cardinality}, $|I_1|\Ge 2$.  
By a similar argument as used earlier, $e_i\in\partial\tilde{V}_1$ and $\psi(e_i)=r_{\sigma^{-1}(i)}e_{\sigma^{-1}(i)}\in\partial\tilde{V}_2$ for all $i=1, \ldots, d-1$. Hence, using $|I_2| \Ge 2$, we have $r_{i}=1$ for all $i=1,\ldots, d-1$. Consequently, by \eqref{dz_i_expression},
\beqn
   b_d\prod_{i\in I_1}b_i^{-\alpha_i}|z_i|^{2\alpha_i}=c_d\prod_{i\in I_2}c_i^{-\beta_i}|z_{\sigma(i)}|^{2\beta_i}, \quad (z_1,z_2,\ldots,z_{d-1}) \in\partial\tilde{V_1}.
   \eeqn
Since for each $j\in\{1,\ldots,d-1\}$, there exists a point $(z_1,z_2,\ldots,z_{d-1})\in\partial\tilde{V_1}$ such that $z_j=0$ and $z_i\neq 0$ for all $i\neq j$, it follows that $\sigma(I_2)=I_1$. This proves the claim. 

By the above claim and \eqref{dz_i_expression}, it follows that 
\beqn
   b_d\prod_{i\in I_1}b_i^{-\alpha_i}|z_i|^{2\alpha_i}=c_d\prod_{i\in I_1}c_{\sigma^{-1}(i)}^{-\beta_{\sigma^{-1}(i)}}|z_{i}|^{2\beta_{\sigma^{-1}(i)}}, \quad (z_1,z_2,\ldots,z_{d-1}) \in\partial\tilde{V_1}.
   \eeqn
We now show that 
\beq
\label{alpha-beta-relation}
\alpha_i=\beta_{\sigma^{-1}(i)}, \quad  i\in I_1.
\eeq
Suppose, to the contrary, that there exists $j \in I_1$ such that $\alpha_j \neq\beta_{\sigma^{-1}(j)}$. Without loss of generality, assume $\alpha_j>\beta_{\sigma^{-1}(j)}$. We show that there exists a sequence of points in $\partial\tilde{V}_1$ such that the modulus of the $j$-th coordinate tends to zero, while the moduli of the remaining coordinates stay bounded below by a positive constant. Given any $\epsilon$, with $1>\epsilon>0$, fix $0< A< \sqrt{\frac{1-\epsilon}{d-1}}$. Let $z_i$ $(i\neq j)$ be complex numbers with $|z_i|\Ge A$ for all $i\neq j$, such that  
\beqn
\sum_{\substack{i=1 \\ i \neq j}}^{d-1} |z_i|^2 = 1 - \epsilon. 
\eeqn
For such a choice, let  $z_j$ be such that $|z_j|$ satisfies 
\beqn
|z_j|^2 + \Big(b_d b_j^{-\alpha_j} \prod_{\substack{i\in I_1 \\ i \neq j}}b_i^{-\alpha_i}|z_i|^{2\alpha_i}\Big) |z_j|^{2\alpha_j} = \epsilon.
\eeqn
Then clearly $|z_j|\in (0, \sqrt{\epsilon})$ and by the definition of $\partial\tilde{V}_1$, it follows that $(z_1,z_2,\dots,z_{d-1})\in \partial\tilde{V}_1$. This shows, in particular, that we can construct a sequence of points in $\partial\tilde{V}_1$ with required properties. By \eqref{dz_i_expression} and using the fact that $\sigma(I_2)=I_1$, such a sequence of points satisfies
\beqn
|z_j|^{2(\alpha_j-\beta_{\sigma^{-1}(j)})}\prod_{\substack{i\in I_1 \\ i \neq j}}|z_i|^{2(\alpha_i-\beta_{\sigma^{-1}(i)})}=\frac{c_d\prod_{i\in I_1}b_i^{\alpha_i}}{b_d\prod_{i\in I_2}c_i^{\beta_i}},
\eeqn
which is a contradiction as the left-hand side of the above identity converges to $0$, but the right-hand side remains a fixed positive constant. This leads to a contradiction, and therefore, we have the claim as in ~\eqref{alpha-beta-relation}.
Finally, using \eqref{alpha-beta-relation} and \eqref{dz_i_expression}, we obtain
\beq\label{last_smp_exp}
c_d \prod_{i\in I_1}b_i^{\alpha_i}=b_d\prod_{i\in I_2}c_i^{\beta_i}.
\eeq
Define a permutation $\tilde{\sigma}$ on the index set $\{1,\ldots,d\}$ that extends $\sigma$, that is,
\beq \label{new-permutation}
\text{$\tilde{\sigma}(i): = \sigma(i)$ for $i=1, \ldots, d-1$ and $\tilde{\sigma}(d): = d$}.
\eeq
By the assumptions on the tuple $\mf n$, note that $\mathcal J(\tilde{\sigma}(\mf n))=\{\ \tilde{\sigma}^{-1}(I_1)\cup\{d\}\}=\{I_2\cup\{d\}\}$, and the unique partition and tuple associated with $I_2\cup\{d\}$ are $\{I_2,\{d\}\}$ and $(\alpha_i)_{i\in I_2\cup\{d\}}=(\beta_i)_{i\in I_2\cup \{d\}}$ where $\alpha_d=\beta_d=1$, respectively, where the equality of tuples follows from \eqref{alpha-beta-relation}.  Similarly, for the tuple $\mf m$,  $\mathcal J(\mf m)=\{I_2\cup\{d\}\}$, and the unique partition and tuple associated with $I_2\cup\{d\}$ are $\{I_2,\{d\}\}$ and $(\beta_i)_{i\in I_2\cup\{d\}},$ respectively.
Furthermore, using  \eqref{last_smp_exp} and \eqref{new-permutation}, we have
\[
c_d\prod_{i\in I_2}b_{\tilde{\sigma}(i)}^{\beta_i}=b_{\tilde{\sigma}(d)}\prod_{i\in I_2}c_i^{\beta_i}.
\]
Hence, the $d$-tuples $(\tilde{\sigma}(\mf b),\tilde{\sigma}(\mf n))$ and $(\mf c,\mf m)$ admit a similar pattern. This completes the proof.
\end{proof}


\section{Application to classical CNP kernels}

In this section, we demonstrate that the results obtained in this paper extend to CNP kernels on domains in $\mathbb C^n$. As a consequence, the theory developed here is applicable to the study of both algebraic and isometric isomorphism problems for the corresponding multiplier algebras. This is achieved by establishing a correspondence between such kernels. We begin by describing how to construct a CNP Dirichlet series kernel from a CNP kernel while keeping the multiplier variety the same.

 Let $k,d\in \mathbb N\cup \{\infty\}$, and let  $\Omega$ be a domain in $\mathbb C^k$. Consider the CNP kernel $K: \Omega \times \Omega \rar \mathbb C$ given by 
\begin{align}\label{classical kernel}
K({\mf z}, {\mf w}) = \frac{1}{1 - \sum\limits_{j=1}^d b_{j} {\mf z}^{\bm \alpha_j}\overline{\mf w}^{\bm\alpha_j}}, \quad \mf z,\mf w\in \Omega,
\end{align}
where $\{\bm\alpha_j: j=1,\dots,d\} \subseteq \mathbb Z_+^k\setminus \{0\}$ and each $b_{j}$ is a positive real number. Observe that the kernel can also be written as
$$
K({\mf z}, {\mf w}) = \frac{1}{1 - \inp{f({\mf z})}{f({\mf w})}_{\mathbb C^d}},
$$
where $f(\mf z) = (\sqrt{b_{j}}{\mf z}^{\bm\alpha_j})_{j=1}^d$.
To such a kernel, one can associate a CNP Dirichlet series kernel as follows.  \emph{Apriori}, this association depends on a choice of primes; however, as a consequence of Theorem ~\ref{var-equal-set}, it follows that this choice does not affect the resulting multiplier algebra, since the associated multiplier varieties coincide. Fix $k$ distinct primes and enumerate them as $\{p_i:i=1,\dots,k\}$. 
For any $\bm\alpha = (\alpha_1, \alpha_2, \ldots, \alpha_k) \in \mathbb Z_+^k$, we set 
\begin{align}\label{multiproduct}
\mf p^{\bm\alpha}: = \prod_{i =1}^k p_i^{\alpha_i}.
\end{align}
With this choice, the associated CNP Dirichlet series kernel corresponding to $K$ is given by 
\begin{align}\label{dirichlet kernel}
K_{\mf b, \mf n}(s, u) = \frac{1}{1 - \sum\limits_{j=1}^d b_{j} (\mf p^{\bm\alpha_j})^{-s- \overline{u}}} = \frac{1}{1 - \inp{f_{\mf b, \mf n}(s)}{f_{\mf b, \mf n}(u)}}_{\mathbb C^d},\quad s,u\in \mathbb H_\rho
\end{align}
where $\mf b = (b_{1}, \ldots, b_{d})$ is the weight data, $\mf n = (\mf p^{\bm\alpha_1}, p^{\bm\alpha_2}, \ldots, \mf p^{\bm\alpha_d})$ is the frequency data, and $\rho\in\mathbb R$ is chosen so that the kernel is well-defined on $\mathbb H_\rho$. The function $f_{\mf b, \mf n}$ is given by 
$$f_{\mf b, \mf n}(s) = (\sqrt{b_{j}}(\mf p^{\bm\alpha_j})^{-s})_{j=1}^d, ~s\in\mathbb H_\rho.$$ One way to choose $\rho$ is to ensure that the point $(p_1^{-\rho}, p_2^{-\rho},\dots, p_k^{-\rho})\in \{\mathbf z\in\mathbb C^k: \sum_{j=1}^db_j|\mf z^{\bm \alpha_j}|^2<1\}$.
We will show that the multiplier varieties associated with $K$ and $K_{\mf b,\mf n}$ coincide. Consequently, their multiplier algebras are isometrically isomorphic. This establishes a bridge between classical CNP kernels and CNP kernels arising from Dirichlet series, allowing results from the latter setting to be applied in the former.

\begin{proposition}\label{same mv}
Let $K$ and $K_{\mf b, \mf n}$ be the CNP kernels as in ~\eqref{classical kernel} and ~\eqref{dirichlet kernel}, respectively. If $V$ and $V_{\mf b,\mf n}$ denote the associated multiplier varieties, then 
\[
V=V_{\mf b, \mf n}.
\]
\end{proposition}
\begin{proof}
Recall that, by the definition of multiplier variety, $V$ is given by
$$
V = \{{\mf z} \in \mathbb B_d : g({\mf z}) = 0~~\text{for all}~~g \in \mathcal M_d~~\text{such that}~~g|_{f(\Omega)} \equiv 0 \}.
$$
Let $g\in\clm_d$ be such that $g|_{f(\Omega)} \equiv 0$. Since $\Omega$ is open and $g \circ f$ vanishes on $\Omega$, it follows by the Identity theorem that $g(f(\mf z))=0$  
for all $\mf z\in \mathbb C^k$ with  
$$\sum_{i=1}^d b_{i} |{\mf z^{\bm\alpha_i}}|^{2} < 1.
$$
Hence $g$ vanishes on  $f_{\mf b, \mf n}(\mathbb H_\rho)$, since \[f_{\mf b, \mf n}(\mathbb H_\rho)\subseteq f\big(\Big\{\mf z\in \mathbb C^k: \sum_{i=1}^d b_{i} |{\mf z}^{\bm\alpha_i}|^{2}<1\Big\}\big).\] 
By definition of multiplier variety, this implies $g|_{V_{\mf b,\mf n}}\equiv0$, and hence $V_{\mf b, \mf n} \subseteq V$. 

 For reverse inclusion, by Theorem~\ref{variety-gen}, $$V_{\mf b, \mf n} = \mathbb B_d \cap \bigcap_{J \in \mathcal J(\mf n)} Z(q_{_J}),$$
where for each $J\in \mathcal J(\mf n)$,
 \begin{equation*}
 q_{_J} (\mf z) = \prod\limits_{i \in J_1}
 (\sqrt{b_{i}})^{\beta_i}\prod\limits_{i \in J_2}z_i^{\beta_i}-\prod\limits_{i \in J_2}(\sqrt{b_{i}})^{\beta_i}\prod\limits_{i \in J_1}z_i^{\beta_i},\quad \mf z = (z_i)_{i=1}^d \in \mathbb C^d.
\end{equation*}
 Recall that for any $J\in \mathcal J(\mf n)$, if $\mf n=(n_1,n_2,\dots,n_d)$, then $(\beta_j)_{j \in J}$ is uniquely determined by
$$\prod_{i \in J_1}	n_i^{\beta_i} = \prod_{i \in J_2} n_i^{\beta_i}.$$ Since in the present context $n_i=\mf p^{\bm\alpha_i},$ the above identity is equivalent to 
$$
\prod_{i\in J_1} \mf p^{\beta_i\bm\alpha_i}=\prod_{i\in J_2} \mf p^{\beta_i\bm\alpha_i}.
$$
Here if $\bm\alpha_i = (\alpha_{i1},\alpha_{i2},\dots, \alpha_{ik})\in \mathbb Z_+^k$, then 
$\beta_i\bm\alpha_i=(\beta_i\alpha_{i1}, \beta_i\alpha_{i2},\dots,\beta_i\alpha_{ik})$. Note that, even when $k=\infty$, then each element $\bm \alpha\in \mathbb Z_+^k$ has finitely many non-zero co-ordinates and therefore, $\mf p^{\bm\alpha}$ reduces to a finite product. Now using ~\eqref{multiproduct} and comparing the powers of different primes, we get 
\beq
\label{m-beta-relation}
\sum_{i \in J_1} \beta_i \alpha_{ij}= \sum_{i \in J_2} \beta_i \alpha_{ij},  \quad j =1, \ldots, k.
\eeq
To complete the proof, it is enough to verify that $f(\Omega) \subseteq V_{\mf b, \mf n}$. So, let ${\mf w}: = (\sqrt{b_{i}} {\mf z}^{\bm\alpha_i})_{i=1}^d \in f(\Omega)$. Then, 
\beqn
  q_{_J}({\mf w}) &=& \prod\limits_{i \in J_1} (\sqrt{b_{\mf m_i}})^{\beta_i}\prod\limits_{i \in J_2} (\sqrt{b_{i}}{\mf z}^{\bm\alpha_i})^{\beta_i}-\prod\limits_{i \in J_2} (\sqrt{b_{i}})^{\beta_i}\prod\limits_{i \in J_1}(\sqrt{b_{i}}{\mf z}^{\bm\alpha_i})^{\beta_i}\\
&=& \Big(\prod\limits_{i \in J} \sqrt{b_{i}}^{\beta_i}\Big) \Big(\prod\limits_{i \in J_2} {\mf z}^{\beta_i\bm\alpha_i} - \prod\limits_{i \in J_1} {\mf z}^{\beta_i\bm\alpha_i}\Big)\\
&=&  \Big(\prod\limits_{i \in J} \sqrt{b_{i}}^{\beta_i} \Big)\Big(\prod_{i\in J_2}\prod_{j=1}^k z_j^{\beta_i \alpha_{ij}} - \prod_{i\in J_1}\prod_{j=1}^k z_j^{\beta_i\alpha_{ij}}\Big) \overset{\eqref{m-beta-relation}}= 0,
\eeqn
which proves that $f(\Omega) \subseteq V_{\mf b, \mf n}$. This completes the proof.
\end{proof}



Now we indicate the converse direction of the correspondence. That is, given a CNP Dirichlet series kernel, we construct a CNP kernel on a domain while keeping the multiplier variety the same. Let $K_{\mf b, \mf n}$ be a CNP Dirichlet series kernel on $\mathbb H_\rho$ with weight data $\mf b = (b_1,\ldots,b_d)$ and frequency data $\mf n = (n_1,n_2,\ldots,n_d)$. Let $p_1,p_2,\ldots,p_k$ be distinct primes that generate all the frequencies; that is,   \[n_i = \mf p^{\bm\alpha_i}=\prod_{j=1}^k p_j^{\alpha_{ij}},\] for some $\bm\alpha_i = (\alpha_{i1}, \alpha_{i2}, \dots, \alpha_{ik}) \in \mathbb Z_+^k$. Define the kernel $K : \Omega \times \Omega \to \mathbb C$ by
$$
K(\mf z, \mf w) = \frac{1}{1 - \sum\limits_{j=1}^d b_j {\mf z}^{\bm\alpha_j} \overline{\mf w}^{\bm\alpha_j}}, \quad \mf z, \mf w \in \Omega,
$$
where 
\[\Omega \subseteq \big\{\mf z \in \mathbb C^k : \sum_{i=1}^d b_{i} |{\mf z}^{\bm\alpha_i}|^2 < 1\big\}\] is a domain. Then it is straightforward to check that the multiplier varieties $V$ and $V_{\mf b, \mf n}$ associated with $K$ and $K_{\mf b, \mf n}$, respectively, coincide. We leave the details to the reader, as they are analogous to those of the preceding proposition. 

We present some examples to illustrate the above discussion.
\begin{example}

\noindent (i) Let $K_1$ and $K_2$ be CNP kernels, as in ~\eqref{classical kernel}, with different number of terms in the denominator; that is, $d$ is different for $K_1$ and $K_2$. Then, by the above discussion, the associated CNP Dirichlet series kernels have frequency data of different lengths. Therefore, by Proposition~\ref{never-isom-iso}, the multiplier algebras associated with $K_1$ and $K_2$ are not isometrically isomorphic.  
\vspace{0.1in}

\noindent (ii)
Let $b_i$ and $c_i$ be positive real numbers for all $i=1,2,3$. Consider the kernels
\[
K_1(\mf z, \mf w)=\frac{1}{1- (b_1 z_1\overline{w_1}+b_2 z_2\overline{w_2}+b_3 z_1 z_2 \overline{w_1 w_2})}, \quad \mf z, \mf w \in \Omega_1,
\]
where $\Omega_1 =\{(z_1, z_2) \in \mathbb C^2 : b_1 |z_1|^2 + b_2 |z_2|^2 + b_3 |z_1 z_2|^2 < 1\}$
and
\[
K_2(\mf z, \mf w)=\frac{1}{1- (c_1 z_1 z_2 \overline{w_1 w_2}+c_2 z_1^2 z_2 \overline{w_1^2 w_2}+c_3 z_1^3 z_2^2 \overline{w_1^3 w_2^2})}, \quad \mf z, \mf w \in\Omega_2,
\]
where $\Omega_2 =\{(z_1, z_2) \in \mathbb C^2 : c_1 |z_1 z_2|^2 + c_2 |z_1^2 z_2|^2 + c_3 |z_1^3 z_2^2|^2 < 1\}.$ 
Let $p_1$ and $p_2$ be different primes. For this choice of primes, the associated CNP Dirichlet series kernels are $K_{\mf b, \mf n}$ and $K_{\mf c, \mf m }$, where 
\[\mf b=(b_1,b_2,b_3),\ \mf n=(p_1, p_2, p_1p_2),\ \mf c=(c_1,c_2,c_3),\ \mbox{ and } \mf m=(p_1p_2, p_1^2p_2, p_1^3p_2^2).\] 
Then, the multiplier varieties associated with $K_{\mf b, \mf n}$ and $K_{\mf c, \mf m}$ are given by
\[
V_1 = \{\mf z \in \mathbb B_3 : \sqrt{b_3} z_1 z_2 = \sqrt{b_1 b_2} z_3\}
\]
and
\[
V_2 = \{\mf z \in \mathbb B_3 : \sqrt{c_3} z_1 z_2 = \sqrt{c_1 c_2} z_3\},
\]
respectively. If, in addition, we assume that $b_3 c_1 c_2 = c_3 b_1 b_2$, then $V_1 = V_2$. Consequently, the multiplier algebras $\clm(\clh(K_1))$ and $\clm(\clh(K_2))$ are isometrically isomorphic.
\end{example}

We conclude the article by showing that even a very special case of Theorem~\ref{isometric isomorphism} yields new results in the classical setting. 
For a $(k+1)$-tuple $\mf b=(b_1,b_2,\dots,b_{k+1})$ of positive real numbers and let $ \alpha \in \mathbb Z_+^k$ have at least two nonzero co-ordinates. We denote by $S_{\mf b,\alpha}$ the CNP kernel 
defined as 
\begin{align}\label{sbm}
S_{\mf b,\alpha}(\mf z,\mf w)=\frac{1}{1-\sum_{i=1}^{k}b_iz_i\overline{w_i}-b_{k+1}\mf z^{\alpha}\overline{\mf w}^{\alpha}}, \quad (\mf z,\mf w\in\Omega),
\end{align}
where 
\[\Omega=\big\{(z_1, z_2,\dots,z_k)\in \mathbb C^k: \sum_{i=1}^k b_i|z_i|^2+ b_{k+1}|\mf z^{\alpha}|^2<1\big\}\] is a domain in $\mathbb C^k$. We write $\mathcal{M}(\mathcal{H}(S_{\mf{b}, \alpha}))$ for the multiplier algebra associated with $S_{\mf b,\alpha}$. 
\begin{theorem}\label{cnp classical}
Let $S_{\mf b, \alpha}$ and $S_{\mf c, \alpha' }$ be CNP kernels as in ~\eqref{sbm}. Then the following are equivalent: 
\begin{itemize}
\item[(i)] $\mathcal{M}(\mathcal{H}(S_{\mf{b}, \alpha}))$ and $\mathcal{M}(\mathcal{H}(S_{\mf{c}, \alpha'}))$ are isomorphic;
\item[(ii)] there exists a permutation $\sigma$ of the index set $\{1,2,\dots,k\}$ such that 
\[
\sigma(\alpha)=\alpha' \mbox{ and } c_{k+1}\prod_{i=1}^kb_{\sigma(i)}^{\alpha_i'}=b_{k+1}\prod_{i=1}^kc_i^{\alpha_i'},
\]
where $\alpha'=(\alpha_1',\alpha_2',\dots, \alpha_k')$;

\item[(iii)] $\mathcal M(\mathcal H(S_{\mf b, \alpha}))$ and $\mathcal M(\mathcal H(S_{\mf c, \alpha'}))$ are isometrically isomorphic. 
\end{itemize} 

\end{theorem}
\begin{proof}
Associated with $S_{\mf b, \alpha}$ and $S_{\mf c, \alpha' }$, we consider the CNP Dirichlet series kernels $K_{\mf b, \mf n}$ and $K_{\mf c, \mf n'}$, respectively, where
\[
\mf n=\big(p_1,p_2,\dots, p_k, \prod_{i=1}^k p_i^{\alpha_i}\big)\ \mbox{and } \mf n'=\big(p_1,p_2,\dots, p_k, \prod_{i=1}^k p_i^{\alpha_i'}\big),\]
for distinct primes $p_1,\dots, p_k$. Then by Proposition~\ref{same mv}, multiplier varieties remain the same. Hence, $\mathcal{M}(\mathcal{H}(S_{\mf{b}, \alpha}))$ and $\mathcal{M}(\mathcal{H}(S_{\mf{c}, \alpha'}))$ are isomorphic (respectively, isometrically isomorphic) if and only if  $\mathcal{M}(\mathcal{H}(K_{\mf{b}, \mf n}))$ and $\mathcal{M}(\mathcal{H}(K_{\mf{c}, \mf n'}))$ are isomorphic (respectively, isometrically isomorphic). 
The result now follows by applying Theorem~\ref{isometric isomorphism} to $K_{\mf b, \mf n}$ and $K_{\mf c, \mf n'}$, and observing that item (ii) is equivalent to the pairs  $(\sigma(\mf b), \sigma(\mf n))$ and $(\mf c, \mf n')$ having a similar pattern.   
\end{proof}

\vspace{0.1in} 

\noindent\textbf{Acknowledgement:} The authors are extremely grateful to Prof. Sushil Gorai for a brief but valuable discussion that eventually led to the proof of Proposition~\ref{negative}. The first named author is supported by Prime Minister's Research Fellowship, ID: 1303115. The second named author's research is supported by the Mathematical Research Impact Centric Support (MATRICS) grant. The third named author acknowledges the support received from the Indian Institute of Technology Bombay through an Institute postdoctoral fellowship.

\vspace{0.1in}
\noindent\textbf{Data Availability:}
Data sharing is not applicable to this article, as no datasets were generated or analyzed during
the current study. In case any datasets are generated during and/or analyzed during the current study, they must
be available from the corresponding author on reasonable request.

\vspace{0.1in}
\noindent\textbf{Conflict of interest:} There is no conflict of interest.

\end{document}